\documentclass[a4paper, 12pt]{amsart}
\usepackage[active]{srcltx}
\usepackage[all]{xy}
\usepackage{subfig}
\usepackage{hyperref}
\hypersetup{
colorlinks,
allcolors=blue
}
\usepackage{mathrsfs}

\usepackage{amsmath,amssymb,amsthm,a4wide}

\usepackage{enumerate}
\usepackage{color}

\theoremstyle{plain}
\newtheorem{lemma}{Lemma}[section]
\newtheorem{theorem}[lemma]{Theorem}

\newtheorem{conjecture}[lemma]{Conjecture}

\theoremstyle{definition}

\newtheorem{definition}[lemma]{Definition}
\newtheorem{remark}[lemma]{Remark}

\theoremstyle{remark}

\numberwithin{equation}{section}

\newcommand{\eps}{\varepsilon}

\newcommand{\be}{\begin{equation}}
\newcommand{\ee}{\end{equation}}

\numberwithin{equation}{section}

\usepackage{listings}
\usepackage{xcolor}

\begin{document}


\title[Closed  Minimal Hypersurfaces]{Closed Minimal Hypersurfaces in $\mathbb{S}^5(1)$ with Constant Scalar and Gauss-Kronecker Curvatures}


\author[Deng]{Qintao Deng$^*$}
\address[Deng]{School of Mathematics and Statistics \& Key Lab NAA--MOE, Central China Normal University, Wuhan 430079, China}
\email{qintaodeng@ccnu.edu.cn}
\thanks{$^{*}$ Supported in part by NSFC (No. 10901067, No. 12271195) and the Special Fund for Basic Scientific Research of Central Colleges (no. CCNU19QN081).}

\author[Kou]{Yunjia Kou}
\address[Kou]{Washington State University, Department of Mathematics and Statistics, Pullman, WA, 99164-3113, USA}
\email{yunjia.kou@wsu.edu}


\date{}



\maketitle

\section*{Abstract}
In this paper, we prove that any closed minimal  hypersurface $M^4$ of $\mathbb{S}^5(1)$ with constant scalar curvature and constant Gauss-Kronecker curvature
must be isoparametric. Specifically, $M^4$ is either an equatorial 4-sphere,  a Clifford torus $\mathbb{S}^2\left(\frac{\sqrt{2}}{2}\right)\times \mathbb{S}^2\left(\frac{\sqrt{2}}{2}\right)$ or $\mathbb{S}^1\left(\frac{1}{2}\right)\times \mathbb{S}^3\left(\frac{\sqrt{3}}{2}\right)$, or a Cartan's minimal hypersurface. Consequently, the squared norm of the second fundamental form $S$ can only take the values 0, 4, 12. This result provides strong support for Chern's Conjecture.

\tableofcontents

\vspace{5mm}

\section{Introduction}

\subsection{Historical Background and Motivation}

The geometry of minimal hypersurfaces in spheres has been a central topic in global differential geometry for decades. A fundamental problem in this field concerns the rigidity and classification of such hypersurfaces under curvature constraints. In \cite{CdS}, S. S. Chern proposed a celebrated conjecture regarding the squared norm of the second fundamental form, denoted by $S$.

\begin{conjecture}[Chern's Conjecture \cite{CdS}]
Let $M^n$ be an $n$-dimensional closed minimal hypersurface in the unit sphere $\mathbb{S}^{n+1}(1)$ with constant scalar curvature. Then the set of possible values for $S$ is discrete.
\end{conjecture}
The following strong version of Chern's conjecture implies the original statement:

\vspace{1.5mm}

\noindent
\textbf{Strong Chern's Conjecture}
\textit{Let $M^n$ be an $n$-dimensional closed minimal hypersurface in the unit sphere $\mathbb{S}^{n+1}(1)$ with constant scalar curvature. Then $M^n$ is isoparametric.}

\vspace{1.5mm}

M\"unzner \cite{M} showed that the number of distinct principal curvatures $g$ of an isoparametric hypersurface in a unit sphere can only be $1, 2, 3, 4$, or $6$. Thus, for a fixed dimension $n$, this rigid structure forces $S$ to take values in a discrete set.

\vspace{1.5mm}

The study of this conjecture was significantly advanced by J. Simons \cite{S} in 1968. He established a fundamental integral inequality which implies the first gap theorem for $S$.


\begin{theorem}\emph{(Simons \cite{S})}
Let $M^n$ be a closed minimal hypersurface in $\mathbb{S}^{n+1}(1)$. Then
\begin{equation}
\int_{M^n}S(S-n) \geq 0,
\end{equation}
where $S$ denotes the squared norm of the second fundamental form.
\end{theorem}

Theorem 1.2 implies that if $0\leq S\leq n$, then either $S\equiv0$ (totally geodesic) or $S\equiv n$. The latter case was completely classified by Chern, do Carmo, and Kobayashi \cite{CdS}, and independently by Lawson \cite{L}, who proved that the Clifford torus are the only closed minimal hypersurfaces in $\mathbb{S}^{n+1}(1)$ satisfying $S=n$.  

Motivated by E. Cartan's isoparametric examples \cite{C}, which suggest that $2n$ might be the next attainable value for $S$, the following \textit{second pinching problem} has attracted considerable attention:

\begin{conjecture}[Second Pinching Conjecture]
Let $M^n$ ($n>3$) be a closed minimal hypersurface in $\mathbb{S}^{n+1}(1)$ with constant scalar curvature. If $S>n$, then $S \geq 2n$.
\end{conjecture}

Significant efforts have been devoted to establishing a gap above $n$. Under the assumption that $S$ is constant, Peng and Terng \cite{PT1} made a pioneering breakthrough, proving that $S > n + \frac{1}{12n}$ whenever $S > n$. This lower bound was subsequently improved to $S > n + \frac{n}{3}$ (for $n>3$) by Cheng and Yang \cite{YC1, YC2, YC3}. To date, the strongest unconditional bound for constant $S$ is $S > n + \frac{3n}{7}$, established by Suh and Yang \cite{SY}.

By imposing additional geometric constraints, one can obtain sharper estimates closer to $2n$. For instance, Cheng and Yang \cite{YC3} proved the following:
\begin{theorem}\emph{(Cheng-Yang \cite{YC3})}\label{YC3}
Let $M^n$ ($n > 3$) be a closed minimal hypersurface in $\mathbb{S}^{n+1}(1)$ with constant scalar curvature. If the cubic form $f_3 = \sum_{i,j,k} h_{ij} h_{jk} h_{ki}\ (=\sum_{i} \lambda_i^3)$ is constant, and $S > n$, then $S > n + \frac{2}{3}n$. The same conclusion holds if $\sum_{i} \lambda_i^4$ is assumed to be constant.
\end{theorem}

Parallel to the study of constant $S$, the generalized pinching problem where $S$ is not assumed constant a priori, has also been extensively investigated. Peng and Terng \cite{PT3} initiated this direction:
\begin{theorem}\emph{(Peng-Terng \cite{PT3})}
Let $M^n$ be a closed minimal hypersurface in $\mathbb{S}^{n+1}(1)$. For $n \leq 5$, there exists a constant $\delta(n)>0$ such that if $n \leq S(x) < n+\delta(n)$ everywhere on $M$, then $S(x) \equiv n$, and $M$ is a Clifford torus.
\end{theorem}
The dimensional restriction was gradually relaxed by Wei and Xu \cite{WX} ($n \leq 7$) and Zhang \cite{Zh} ($n \leq 8$). Eventually, Ding and Xin \cite{DX} completely removed the dimensional restriction, proving the theorem for all $n$ with a gap $\delta(n) \ge \frac{n}{23}$. This bound was further refined to $\frac{n}{18}$ by Xu-Xu \cite{XX} and Lei-Xu-Xu \cite{LXX}.

\vspace{3mm}

For the specific case of $n=3$, the conjecture has been fully resolved. A cornerstone in this area is the celebrated rigidity result by de Almeida and Brito \cite{AB}, who proved that any closed minimally immersed hypersurface $M^3$ with non-negative constant scalar curvature $R_M$ in $\mathbb{S}^4(1)$ is necessarily isoparametric.
By settling the complementary case of the Peng-Terng theorem \cite{PT1}, Chang \cite{Ch} provided a complete classification:
\begin{theorem}\emph{(Chang \cite{Ch})}
A closed minimal hypersurface with constant scalar curvature in $\mathbb{S}^4(1)$ is isoparametric. Specifically, it is either an equatorial 3-sphere ($S=0$), a Clifford torus ($S=3$), or a Cartan minimal hypersurface ($S=6$).
\end{theorem}

However, if the compactness assumption is dropped, the classification remains a challenging open problem, known as Bryant's Conjecture \cite{Ch}:

\begin{conjecture}[Bryant's Conjecture]
A complete minimal hypersurface of constant scalar curvature in $\mathbb{S}^4(1)$ is isoparametric.
\end{conjecture}

For dimension $n=4$, progress has historically relied on additional algebraic or geometric assumptions. Tang and Yang \cite{TY} proved a rigidity theorem assuming constant $f_3$ and a constant number of distinct principal curvatures. Latter, Deng, Gu, and Wei \cite{DGW} utilized delicate local analysis to remove the non-negative scalar curvature assumption in the work of Lusala, Scherfner, and Sousa Jr. \cite{LSS}, obtaining a classification for Willmore hypersurfaces.

Spruck and XIao \cite{SX} established a classification for closed minimal hypersurfaces $M^4$ in $\mathbb{S}^5(1)$ with constant $f_3$ and non-negative constant scalar curvature. Most notably, He, Xu, and Zhao \cite{HXZ} recently removed the non-negative scalar curvature assumption, thereby completing the classification for $n=4$ under the sole condition of a constant $f_3$:
\begin{theorem}\emph{(He-Xu-Zhao \cite{HXZ})}
Any closed minimal hypersurface $M^4$ in $\mathbb{S}^5(1)$ with constant scalar curvature and constant $f_3$ must be isoparametric. Specifically, $M^4$ is either an equatorial 4-sphere, a Clifford torus, or a Cartan minimal hypersurface. Consequently, $S\in\{0, 4, 12\}$.
\end{theorem}

Alternatively, instead of restricting $f_3$, Cui \cite{Cui} investigated the vanishing Gauss-Kronecker curvature case:
\begin{theorem}\emph{(Cui \cite{Cui})}\label{Cui}
Let $M^4$ be a closed minimal hypersurface in $\mathbb{S}^5(1)$ with constant scalar curvature and zero Gauss-Kronecker curvature. Then $M$ is totally geodesic.
\end{theorem}

Beyond the minimal case, analogous rigidity results hold for hypersurfaces with constant mean curvature. Almeida and Brito \cite{AB}, together with Chang \cite{Ch1}, proved that a closed hypersurface $M^3\subset \mathbb{S}^4(1)$ with constant mean curvature and constant scalar curvature is isoparametric. This was later generalized by Almeida, Brito, and Sousa Jr. \cite{ABS}, and extended to higher dimensions by Tang, Wei, and Yan \cite{TWY}. More recently, Tang and Yan \cite{TYan} showed that the constancy of the first $(n-1)$ power sums of principal curvatures already implies isoparametricity:

\begin{theorem}[Tang--Yan \cite{TYan}]
Let $M^n$ $(n>3)$ be a closed hypersurface in $\mathbb{S}^{n+1}(1)$. Assume that
$R_M \ge 0$ and $\sum_{i=1}^n \lambda_i^k$ is constant for each $k=1,\ldots,n-1$,
where $\lambda_1,\ldots,\lambda_n$ are the principal curvatures of $M$. Then $M$ is isoparametric. Moreover, if $M$ has $n$ distinct principal curvatures at some point, then $R_M \equiv 0$.
\end{theorem}


\subsection{Main Results}

In this paper, we focus on closed minimal hypersurfaces in $\mathbb{S}^{5}(1)$ with constant scalar curvature. While He, Xu, and Zhao \cite{HXZ} approached this classification by assuming a constant $f_3$, it is also geometrically natural to ask whether replacing this with the assumption of constant Gauss-Kronecker curvature is sufficient to force the hypersurface to be isoparametric.
Our main result gives an affirmative answer to this question, which improves upon Theorem \ref{Cui} and provides further evidence for Chern's Conjecture:
\begin{theorem}[Main Theorem]
Let $M^4$ be a closed minimal hypersurface in $\mathbb{S}^5(1)$ with constant scalar curvature and constant Gauss-Kronecker curvature. Then $M^4$ must be isoparametric. Specifically, $M^4$ is either:
\begin{enumerate}
    \item an equatorial 4-sphere ;
    \item a Clifford torus $\mathbb{S}^2\left(\frac{\sqrt{2}}{2}\right)\times \mathbb{S}^2\left(\frac{\sqrt{2}}{2}\right)$ or $\mathbb{S}^1\left(\frac{1}{2}\right)\times \mathbb{S}^3\left(\frac{\sqrt{3}}{2}\right)$;
    \item or a Cartan minimal hypersurface.
\end{enumerate}
In particular, the squared norm of the second fundamental form $S$ can only be $0, 4, 12$.
\end{theorem}

Importantly, we do not a priori assume a non-negative scalar curvature or a constant number of distinct principal curvatures. We overcome the former by constructing suitable globally defined weighted 3-forms that have a definite sign upon integration, coupled with a rigidity relation between $R_M$ and $S$ when $g=4$. The latter challenge of dealing with mixed cases, is then resolved via a localized cut-off argument.

Very recently, an independent concurrent work by Ge, Liu, Luo, and Yan \cite{GLLY} established the same classification theorem. Although both works share one differential 3-form corresponding to a common weight, the subsequent methodologies are fundamentally different. The proof in \cite{GLLY} partitions the manifold via $f_3$, introduces a second weight generated by $x^{-1}$, and employs global topological tools: the Euler characteristic. By contrast, our proof works directly with the principal curvatures via another shifted 3-form, which is based on direct algebraic estimates for homogeneous polynomials and a cut-off argument.

\vspace{2mm}

The remainder of this paper is organized as follows. In Section 2, we review the fundamental formulas for minimal hypersurfaces in spheres. In Section 3, we analyze the case where there exists a point with exactly two distinct principal curvatures, proving that $M^4$ must be a Clifford torus. Section 4 is devoted to the case where there are four distinct principal curvatures everywhere, where we show that $S=12$ thus $M^4$ is Cartan's minimal hypersurface. In Section 5, we exclude the possibility of having exactly three distinct principal curvatures by deriving a contradiction with the results of Section 4. Finally, the proof of the Main Theorem is completed in Section 6.


\section{Basic formulas for closed minimal hypersurfaces in $\mathbb{S}^{n+1}(1)$}

In this section, we recall some basic formulas for closed minimal hypersurfaces in $\mathbb{S}^{n+1}(1)$, which can be found in \cite{Cui, LSS, PT1, PT3}.

Let $M^n$ be an $n$-dimensional closed hypersurface with constant mean curvature $H$ in $\mathbb{S}^{n+1}(1)$. We choose a local orthonormal frame field $\{e_1, e_2, \cdots, e_{n+1}\}$ such that $\{e_1, e_2, \cdots, e_{n}\}$ is tangent to $M^n$. Let $h_{ij}$ and $H$ denote the components of the second fundamental form and the mean curvature respectively. Then
$$H=\frac{1}{n}\sum_{i}h_{ii}, \quad S=\sum_{i,j}h_{ij}^2, \quad f_3=\sum_{i,j,k}h_{ij}h_{jk}h_{ki}, \quad f_4=\sum_{i,j,k,l}h_{ij}h_{jk}h_{kl}h_{li}.$$

At any fixed point $p\in M$, we can choose an orthonormal frame such that $h_{ij}=\lambda_i\delta_{ij}$ for all $i,j$. Then, at this point $p$, we have
$$H=\frac{1}{n}\sum_{i=1}^n \lambda_i, \quad S=\sum_{i=1}^n \lambda_i^2, \quad f_3=\sum_{i=1}^n \lambda_i^3, \quad f_4=\sum_{i=1}^n \lambda_i^4.$$


Let $h_{ijk}$ and $h_{ijkl}$ be the components of the first and second covariant derivatives of the second fundamental form, respectively. Define $A$ and $B$ by
$$A=\sum_{i,j,k}h_{ijk}^2 \lambda_i^2 \quad \text{and} \quad B=\sum_{i,j,k}h_{ijk}^2 \lambda_i\lambda_j.$$
A straightforward computation gives the following formulas for closed minimal hypersurfaces with constant scalar curvature:
\begin{equation}\label{eq:b1}
\sum_{i,j,k}h_{ijk}^2=S(S-n),
\end{equation}


\begin{equation}\label{eq:b3}
\Delta f_{4}=4 \big((n-S)f_{4}+2A+B\big).
\end{equation}

\vspace{2mm}

Denoting the Gauss-Kronecker curvature by $\mathcal{K} = \lambda_1 \lambda_2 \lambda_3 \lambda_4$, we obtain the following relation between $f_4$ and $\mathcal{K}$ from Newton's identities:
\begin{equation}\label{f_4:K}
    f_4=\frac{1}{2}S^2-4\mathcal{K}.
\end{equation}
Therefore, the assumption that the Gauss-Kronecker curvature $\mathcal{K}$ is constant is equivalent to $f_4 \equiv \text{constant}$.

\vspace{2mm}

From the Gauss equation, the curvature tensor of $M^4$ in $\mathbb{S}^5(1)$ is given by
\begin{equation}\label{Rijkl}
    R_{ijkl} = \delta_{ik} \delta_{jl} - \delta_{il} \delta_{jk} + h_{ik} h_{jl} - h_{il} h_{jk},
\end{equation}
where $\delta_{ij}$ is the Kronecker symbol. In particular, when $i\ne j$, we have
\[R_{ijij} = 1 + \lambda_i \lambda_j.\]
This relation will be used frequently in the subsequent discussion.

\vspace{5mm}

\section{Two distinct principal curvatures at one point}
In this section, we consider the case where there exists a point on the hypersurface with exactly two distinct principal curvatures. More precisely, we prove the following result.
\begin{theorem}\label{theorem3}
    Let $M^4 \hookrightarrow \mathbb{S}^5(1)$ be a closed minimal hypersurface with constant scalar curvature and constant Gauss-Kronecker curvature. If there exists a point with exactly two distinct principal curvatures, then $S=4$ and $M^4$ is the Clifford torus $\mathbb{S}^1\left(\frac{1}{2}\right)\times \mathbb{S}^3\left(\frac{\sqrt{3}}{2}\right)$ or $\mathbb{S}^2\left(\frac{\sqrt{2}}{2}\right)\times \mathbb{S}^2\left(\frac{\sqrt{2}}{2}\right)$.
\end{theorem}

Unless otherwise specified, throughout this section we assume that $M^4$ is a closed minimal hypersurface in $\mathbb{S}^5(1)$ with constant scalar curvature and constant Gauss-Kronecker curvature. We also assume the principal curvatures are ordered as $\lambda_1 \le \lambda_2 \le \lambda_3 \le \lambda_4$. Moreover, all subsequent local computations are performed at the aforementioned point.

\vspace{2mm}

We begin by considering the first covariant derivative of the mean curvature. Since $H \equiv 0$, taking the first derivative shows
\begin{equation}\label{3.1_dS}
h_{11k}+h_{22k}+h_{33k}+h_{44k}=0, \quad k=1,2,3,4.
\end{equation}
Our argument will then be divided into the following two cases based on the multiplicities of the principal curvatures.

\addtocontents{toc}{\protect\setcounter{tocdepth}{-1}} 
\subsection*{Case I. $\lambda_1=\lambda_2=\lambda_3<\lambda_4$}
\begin{lemma}\label{lemma3.1}
Let $M^4 \hookrightarrow \mathbb{S}^5(1)$ be a closed minimal hypersurface with constant scalar curvature and constant Gauss-Kronecker curvature. If there exists a point $p \in M^4$ at which $M^4$ has a principal curvature of multiplicity three $\lambda_1 = \lambda_2 = \lambda_3$ and a simple principal curvature $\lambda_4$, then $S=4$ and $M^4$ is the Clifford torus $\mathbb{S}^1\left(\frac{1}{2}\right)\times \mathbb{S}^3\left(\frac{\sqrt{3}}{2}\right)$.
\end{lemma}
\begin{proof}
\textbf{Step 1:} WLOG, we may assume that at the point $p$, the principal curvatures satisfy $\lambda_1=\lambda_2=\lambda_3=:\lambda<0$ and $\lambda_4 = -3\lambda > 0$ (otherwise, we may simply reverse the orientation). Consequently, with respect to the principal directions $\{e_1(p), e_2(p), e_3(p), e_4(p)\}$, the matrix of the second fundamental form is given by 
$$\Big(h_{ij}(p)\Big)=\left(\begin{matrix}\lambda &\ &\ &\ \\
\  &\lambda&\ &\ \\ \  &\ &\lambda &\ \\  \  &\ &\ &-3\lambda \\
\end{matrix}\right).$$
A straightforward calculation provides
\begin{equation}\label{S:f4}
    S\equiv 12\lambda^2 >0,\ \ f_4 \equiv \dfrac{7}{12}S^2.
\end{equation}

\vspace{2mm}

Since $S$ is constant, differentiating $S$ gives $h_{ij}h_{ijk}=0$. Specifically, evaluating this at $p$ provides \[\lambda h_{11k}+\lambda h_{22k}+\lambda h_{33k}-3\lambda h_{44k}=0.\]
Combined with \eqref{3.1_dS}, this implies that at $p$, $h_{44k}=0$ for $k=1,2,3,4.$ 
\vspace{0.15cm}

 \textbf{Step 2:} To apply \eqref{eq:b3}, we first evaluate the quantity $(S-4) f_4$.
 
 By applying an appropriate orthogonal transformation to the eigenspace associated with $\lambda$, WLOG, we may assume that at $p$, $h_{124}=h_{134}=h_{234}=0.$
Denoting
\[a=\sum_{i=1}^3h_{ii4}^2(p),\ b=\sum_{i,j,k=1}^3 h_{ijk}^2(p),\]
it follows from \eqref{eq:b1} that 
\[S(S-4)=\sum_{i,j,k}h_{ijk}^2=3a+b.\]
Utilizing \eqref{S:f4}, we obtain 
\begin{equation}\label{(S-4(f_4)}
(S-4)f_4 = \frac{7}{12}S^2 (S-4)=\frac{7S}{12}(3a+b).
\end{equation}

\vspace{5mm}

Next, by the total symmetry of $h_{ijk}$, the computation of $2A+B$ at $p$ can be simplified as follows: 
\begin{align}\label{2A+B}
2A + B &= \frac{1}{3}\left(\sum_{i,j,k=1}^4 (2\lambda_i^2 + 2\lambda_j^2 + 2\lambda_k^2+\lambda_i\lambda_j + \lambda_j\lambda_k + \lambda_k\lambda_i) h_{ijk}^2\right), \nonumber\\[2mm]
 &=\frac{1}{3}\left(3\sum_{j,k=1}^3 17\lambda^2 h_{4jk}^2 + \sum_{i,j,k=1}^3 9\lambda^2 h_{ijk}^2\right) \nonumber\\[2mm]
 &= 17\lambda^2 a + 3\lambda^2 b = S\left(\frac{17}{12}a+\frac{1}{4}b\right).
\end{align}
    
    
    

\vspace{1.5mm}

Finally, invoking \eqref{eq:b3} and $\Delta f_4 = 0$, equations \eqref{(S-4(f_4)} and \eqref{2A+B} yield
\[S\left(\frac{17}{12}a+\frac{1}{4}b\right) = \frac{7S}{12}(3a+b),\]
which implies $a=-b$ at $p$.

\vspace{1mm}

Since $a \geq 0$ and $b \geq 0$, the condition $a = -b$ necessitates $a=b=0$; that is, at $p$,
\[h_{ijk}=0,\ \ \forall i,j,k=1,2,3,4.\]
Applying \eqref{eq:b1}, we conclude that $S=4$, thus $M^4$ is  the Clifford torus $\mathbb{S}^1\left(\frac{1}{2}\right)\times \mathbb{S}^3\left(\frac{\sqrt{3}}{2}\right)$. 

\end{proof}

\subsection*{Case II. $\lambda_1 = \lambda_2 < \lambda_3 =\lambda_4$}

\vspace{2mm}
\begin{lemma}\label{lemma3.2}
Let $M^4 \hookrightarrow \mathbb{S}^5(1)$ be a closed minimal hypersurface with constant scalar curvature and constant Gauss-Kronecker curvature. If there exists a point $p \in M^4$ where the principal curvatures have multiplicities $2$ and $2$ with $\lambda_1 = \lambda_2 < \lambda_3 = \lambda_4$, then $S=4$ and $M^4$ is the Clifford torus $\mathbb{S}^2\left(\frac{\sqrt{2}}{2}\right) \times \mathbb{S}^2\left(\frac{\sqrt{2}}{2}\right)$.
\end{lemma}
\begin{proof}
\textbf{Step 1:} Suppose that at a point $p$, with respect to the frame $\{e_1(p), e_2(p), e_3(p), e_4(p)\}$, the matrix of the second fundamental form is given by
$$\Big(h_{ij}(p)\Big)=\left(\begin{matrix}\lambda &\ &\ &\ \\
\  &\lambda&\ &\ \\ \  &\ &-\lambda &\ \\  \  &\ &\ &-\lambda \\
\end{matrix}\right).$$
A direct calculation gives $S=4\lambda^2 >0$ and $f_4=\dfrac{1}{4}S^2.$

\vspace{2mm}

\textbf{Step 2:} Similar to Case I, we now evaluate the quantities $(S-4)f_4$ and $2A+B$ in \eqref{eq:b3}.

First, we define the index sets $I_1=\{1,2\}$ and $I_2=\{3,4\}$. Let
\[a :=\sum_{i,j\in I_1,\ k\in I_2}h_{ijk}^2+\sum_{i,j\in I_2,\ k\in I_1}h_{ijk}^2,\ \ \ b := \sum_{i,j,k\in I_1}h_{ijk}^2+\sum_{i,j,k\in I_2}h_{ijk}^2.\]
Then, it follows from \eqref{eq:b1} that 
\[S(S-4)=\sum_{i,j,k}h_{ijk}^2=3a+b,\]
which implies
\begin{equation}\label{(S-4)f_4:2}
    (S-4)f_4 = \frac{1}{4} S^2(S-4) = \frac{S}{4} (3a+b).
\end{equation}

Next, to compute $2A+B$ at $p$, we partition the sum over all indices $i, j, k$ into four disjoint groups based on the subsets $I_1$ and $I_2$. For each configuration, similar to Case I, we evaluate the value $2A+B$:
\begin{align}\label{2A+B:2}
2A + B &= \frac{1}{3}\left(\sum_{i,j,k=1}^4 (2\lambda_i^2 + 2\lambda_j^2 + 2\lambda_k^2+\lambda_i\lambda_j + \lambda_j\lambda_k + \lambda_k\lambda_i) h_{ijk}^2\right) \nonumber\\[2mm]
 &= \frac{1}{3}\left[ 3 \left(\sum_{ i,j \in I_1, k \in I_2}  5\lambda^2 h_{ijk}^2 + \sum_{i,j \in I_2, k \in I_1 } 5\lambda^2 h_{ijk}^2 \right)
 + \sum_{i,j,k \in I_1 }9\lambda^2 h_{ijk}^2 + \sum_{i,j,k \in I_2 } 9\lambda^2 h_{ijk}^2 \right]\nonumber \\[2mm]
 &= 5\lambda^2 a + 3\lambda^2 b = \frac{S}{4}(5a+3b).
\end{align}

Finally, substituting \eqref{(S-4)f_4:2} and \eqref{2A+B:2} into \eqref{eq:b3}, and using $\Delta f_4 = 0$, we find
\[\frac{S}{4}\left(5a+3b\right) = \frac{S}{4}(3a+b),\]
which forces $a=-b$ at $p$. 

Since $a \geq 0$ and $b \geq 0$, the equality $a = -b$ shows $a=b=0$. As in case I, this implies $h_{ijk}=0$ for all $i,j,k$ at $p$. Applying \eqref{eq:b1} once more, we conclude that $S = 4$, thus $M^4$ is the Clifford torus $\mathbb{S}^2\left(\frac{\sqrt{2}}{2}\right) \times \mathbb{S}^2\left(\frac{\sqrt{2}}{2}\right)$.
\end{proof}

Theorem \ref{theorem3} follows directly from Lemmas \ref{lemma3.1} and \ref{lemma3.2}.


\vspace{5mm}

\addtocontents{toc}{\protect\setcounter{tocdepth}{2}}
\section{Four distinct principal curvatures everywhere}
In this section, we examine the case where the hypersurface $M^4$ possesses four distinct principal curvatures at every point. With the principal curvatures ordered as $\lambda_1 < \lambda_2 < \lambda_3 < \lambda_4$, we obtain the following result:
\begin{theorem}\label{th4}
    Let $M^4\hookrightarrow \mathbb{S}^5(1)$ be a closed minimal hypersurface with constant scalar curvature and constant Gauss-Kronecker curvature. If there are 4 distinct principal curvatures everywhere on $M^4$, then $S=12$ and $M^4$ is Cartan's minimal hypersurface.
\end{theorem}
The proof in this section will follow the outline: (i) the Gauss-Kronecker curvature $\mathcal{K}$ must be positive; (ii) the scalar curvature is non-negative, which is equivalent to $S\le 12$; (iii) $S=12$ and $M^4$ is isoparametric. 


\vspace{3mm}
We begin by establishing a useful preliminary result.
\begin{lemma}\label{lemmaGB}
Let $M^4\hookrightarrow \mathbb{S}^5(1)$ be a closed minimal hypersurface with constant scalar curvature and constant Gauss-Kronecker curvature. If $M^4$ admits a globally simple principal curvature, then 
\begin{equation}\label{S:K}
S=6(\mathcal{K}+1).
\end{equation}
\end{lemma}

\begin{proof}
According to Lemma 6 in \cite{Cui}, the Gauss-Bonnet theorem leads to
\[\int_{M^4}\left(\frac{3}{2}S^2-3f_4-2S+12\right)dV = 16\pi^2\chi(M^4).\]
(For a detailed derivation of this equation, please see \cite{Cui}.)

\vspace{1.5mm}

Since the eigenvector field corresponding to the simple eigenvalue has no isolated zeros, the Poincar\'e-Hopf lemma implies the Euler characteristic $\chi(M^4) = 0$. Because both $S$ and $f_4$ are constant, the integrand must vanish identically, i.e.
\begin{equation}\label{f_4v.s.S}
f_4=\frac{1}{2}S^2-\frac{2}{3}S+4.
\end{equation}
Combining this with equation \eqref{f_4:K}, we obtain \eqref{S:K}.
\end{proof}

Furthermore, note that equation \eqref{eq:b1} implies $S \ge 4$. However, if $S=4$, then $M^4$ must be a Clifford torus, which does not possess four distinct principal curvatures everywhere. Consequently, throughout this section, we must strictly have $S > 4$.

\subsection{The Gauss-Kronecker curvature is positive}

\begin{lemma}\label{lemma4.3:K>0}
    Let $M^4 \hookrightarrow \mathbb{S}^5(1)$ be a closed minimal hypersurface with constant scalar curvature and constant Gauss-Kronecker curvature. If there are 4 distinct principal curvatures everywhere on $M^4$, then the Gauss-Kronecker curvature must be positive.
\end{lemma}
\begin{proof}

We argue by contradiction. Suppose that the constant Gauss-Kronecker curvature is non-positive (i.e. $\mathcal{K} \le 0$). Then, \eqref{S:K} in Lemma \ref{lemmaGB} implies that
\[S=6(\mathcal{K}+1) \le 6.\]
However, since we have already established $S > 4$, thanks to Theorem \ref{YC3} (Yang-Cheng \cite{YC3}), it ensures that 
\[S > n + \frac{2}{3}n = \frac{20}{3}.\]
This stands in contradiction to $S \le 6$, completing the proof.
\end{proof}

\begin{remark}
    With the principal curvatures ordered as $\lambda_1 < \lambda_2 < \lambda_3 < \lambda_4$, Lemma \ref{lemma4.3:K>0} implies that they must be distributed as follows:
    \[\lambda_1 < \lambda_2 < 0 < \lambda_3 < \lambda_4.\]
    Because the Gauss-Kronecker curvature $\mathcal{K}$ is a positive constant, this distribution holds everywhere on $M^4$. 
\end{remark}



\subsection{The scalar curvature is non-negative ($S\le 12$)}
\begin{lemma}\label{lemmaS>12}
    Let $M^4 \hookrightarrow \mathbb{S}^5(1)$ be a closed minimal hypersurface with constant scalar curvature and constant Gauss-Kronecker curvature. If there are 4 distinct principal curvatures everywhere on $M^4$, then the scalar curvature must be non-negative, which is equivalent to $S\le 12$.
\end{lemma}
For the sake of contradiction, we suppose $S>12$ through this subsection.

\subsubsection*{Preliminary computations}
Since $\lambda_1 < \lambda_2 < \lambda_3 < \lambda_4$ everywhere and $H, S$, and $f_4$ are constants, differentiating gives the equation:
\[\begin{cases}
 h_{11k} + h_{22k} + h_{33k} + h_{44k} &= 0\\[2mm]
 \lambda_1 h_{11k} + \lambda_2 h_{22k} + \lambda_3h_{33k} + \lambda_4 h_{44k}&= 0\\[2mm]
\lambda_1^3 h_{11k} + \lambda_2^3 h_{22k} + \lambda_3^3 h_{33k} + \lambda_4^3 h_{44k}& = 0.
\end{cases}\]
Because the $\lambda_i$ are distinct, we can apply Cramer's rule to solve for $k = 1, 2, 3, 4$, obtaining:
\begin{equation}\label{solu}
    \begin{aligned}
        h_{11k} &= - \left( \frac{\lambda_1}{\lambda_4} \right) \frac{(\lambda_3 - \lambda_4)(\lambda_2 - \lambda_4)}{(\lambda_3 - \lambda_1)(\lambda_2 - \lambda_1)} h_{44k},\\[2mm]
        h_{22k} & =  \left( \frac{\lambda_2}{\lambda_4} \right) \frac{(\lambda_4 - \lambda_3)(\lambda_4 - \lambda_1)}{(\lambda_2 - \lambda_1)(\lambda_3 - \lambda_2)} h_{44k},\\[2mm]
        h_{33k} &= - \left( \frac{\lambda_3}{\lambda_4} \right) \frac{(\lambda_4 - \lambda_1)(\lambda_4 - \lambda_2)}{(\lambda_3 - \lambda_1)(\lambda_3 - \lambda_2)} h_{44k}.
    \end{aligned}
\end{equation}
Also for $1\le i < j \le 4$, 
\[R_{ijij} = 1 + \lambda_i \lambda_j,\ \ \omega_{ij} = \sum_{k=1}^4 \frac{h_{ijk} \omega_k}{\lambda_i - \lambda_j}. \]
Now we borrow the same 3-form defined as in \cite{SX}:
\begin{definition} \label{theta:Phi}
    Let
\begin{align*}
\theta_{12} &= \omega_3 \wedge \omega_4 \wedge \omega_{12},\  
\theta_{13} = \omega_4 \wedge \omega_2 \wedge \omega_{13}, \ 
\theta_{14} = \omega_2 \wedge \omega_3 \wedge \omega_{14}, \\
\theta_{23} &= \omega_1 \wedge \omega_4 \wedge \omega_{23}, \ \theta_{24} = \omega_3 \wedge \omega_1 \wedge \omega_{24},\ 
\theta_{34} = \omega_1 \wedge \omega_2 \wedge \omega_{34},
\end{align*}
and define the basic 3-form \[
\Phi = \sum_{i<j} \theta_{ij}.
\]
\end{definition} 
To compute $d\Phi$, by virtue of symmetry, it suffices to calculate $d\theta_{12}$ and apply the corresponding permutations. Alternatively, we refer to Lemma 3.7 in \cite{SX} for the complete explicit calculations.

\begin{align}\label{theta12}
d\theta_{12} 
&= (d\omega_3)\wedge \omega_4 \wedge \omega_{12}
- \omega_3 \wedge (d\omega_4)\wedge \omega_{12}
+ \omega_3 \wedge \omega_4 \wedge (d\omega_{12}) \nonumber \\[2mm]
&= \omega_{3k}\wedge \omega_k \wedge \omega_4 \wedge \omega_{12}
- \omega_3 \wedge \omega_{4k}\wedge \omega_k \wedge \omega_{12} \nonumber
\\[1mm]
& \quad + \omega_3 \wedge \omega_4 \wedge (\omega_{1k}\wedge \omega_{k2}
- R_{1212}\,\omega_1 \wedge \omega_2) \nonumber \\[2mm]
&= \omega_{31}\wedge \omega_1 \wedge \omega_4 \wedge \omega_{12}
+ \omega_{32}\wedge \omega_2 \wedge \omega_4 \wedge \omega_{12}
- \omega_3 \wedge \omega_{41}\wedge \omega_1 \wedge \omega_{12} \nonumber\\[1mm]
&\quad  - \omega_3 \wedge \omega_{42}\wedge \omega_2 \wedge \omega_{12}
+ \omega_3 \wedge \omega_4 \wedge \omega_{13}\wedge \omega_{32} + \omega_3 \wedge \omega_4 \wedge \omega_{14}\wedge \omega_{42} - R_{1212} \mathrm{vol} \nonumber\\[2mm]
& = \vartheta_1 + \vartheta_2 - \vartheta_3 - \vartheta_4 + \vartheta_5 + \vartheta_6 - R_{1212} \mathrm{vol}.    
\end{align}
The forms denoted by $\vartheta_1, \dots, \vartheta_6$ expand as follows:
\begin{align}
\vartheta_1
&=\left(\frac{h_{312}\omega_2 + h_{313}\omega_3}{\lambda_3 - \lambda_1}\right)
\wedge \omega_1 \wedge \omega_4 \wedge
\left(\frac{h_{122}\omega_2 + h_{123}\omega_3}{\lambda_1 - \lambda_2}\right)  \nonumber\\[2mm]
& = \left(
\frac{h_{123}^2}{(\lambda_3 - \lambda_1)(\lambda_1 - \lambda_2)}
- \frac{h_{331}h_{221}}{(\lambda_3 - \lambda_1)(\lambda_1 - \lambda_2)}
\right) \mathrm{vol}.
\end{align}

\begin{align}
\vartheta_2
&= \left(\frac{h_{321}\omega_1 + h_{323}\omega_3}{\lambda_3 - \lambda_2}\right)
\wedge \omega_2 \wedge \omega_4 \wedge
\left(\frac{h_{123}\omega_3 + h_{121}\omega_1}{\lambda_1 - \lambda_2}\right) \nonumber \\[2mm]
&= \left(
-\frac{h_{123}^2}{(\lambda_3 - \lambda_2)(\lambda_1 - \lambda_2)}
+ \frac{h_{332}h_{112}}{(\lambda_3 - \lambda_2)(\lambda_1 - \lambda_2)}
\right) \mathrm{vol}.
\end{align}

\begin{align}
\vartheta_3
&= \omega_3 \wedge 
\left(\frac{h_{412}\omega_2 + h_{414}\omega_4}{\lambda_4 - \lambda_1}\right)
\wedge \omega_1 \wedge
\left(\frac{h_{124}\omega_4 + h_{122}\omega_2}{\lambda_1 - \lambda_2}\right) \nonumber\\[2mm]
&= \left(
-\frac{h_{124}^2}{(\lambda_4 - \lambda_1)(\lambda_1 - \lambda_2)}
+ \frac{h_{441}h_{221}}{(\lambda_4 - \lambda_1)(\lambda_1 - \lambda_2)}
\right)  \mathrm{vol}.
\end{align}
\begin{align}
\vartheta_4
&= \omega_3 \wedge 
\left(\frac{h_{421}\omega_1 + h_{424}\omega_4}{\lambda_4 - \lambda_2}\right)
\wedge \omega_2 \wedge
\left(\frac{h_{124}\omega_4 + h_{121}\omega_1}{\lambda_1 - \lambda_2}\right) \nonumber\\[2mm]
&= \left(
\frac{h_{124}^2}{(\lambda_4 - \lambda_2)(\lambda_1 - \lambda_2)}
- \frac{h_{442}h_{112}}{(\lambda_4 - \lambda_2)(\lambda_1 - \lambda_2)}
\right) \mathrm{vol}.
\end{align}

\begin{align}
\vartheta_5
&= \omega_3 \wedge \omega_4 \wedge
\left(\frac{h_{132}\omega_2 + h_{131}\omega_1}{\lambda_1 - \lambda_3}\right)
\wedge
\left(\frac{h_{321}\omega_1 + h_{322}\omega_2}{\lambda_3 - \lambda_2}\right) \nonumber\\[2mm]
&= \left(
-\frac{h_{123}^2}{(\lambda_1 - \lambda_3)(\lambda_3 - \lambda_2)}
+ \frac{h_{113}h_{223}}{(\lambda_1 - \lambda_3)(\lambda_3 - \lambda_2)}
\right)  \mathrm{vol}.
\end{align}

\begin{align}
\vartheta_6
&= \omega_3 \wedge \omega_4 \wedge
\left(\frac{h_{141}\omega_1 + h_{142}\omega_2}{\lambda_1 - \lambda_4}\right)
\wedge
\left(\frac{h_{421}\omega_1 + h_{422}\omega_2}{\lambda_4 - \lambda_2}\right) \nonumber\\[2mm]
&= \left(
-\frac{h_{124}^2}{(\lambda_1 - \lambda_4)(\lambda_4 - \lambda_2)}
+ \frac{h_{114}h_{224}}{(\lambda_1 - \lambda_4)(\lambda_4 - \lambda_2)}
\right)  \mathrm{vol}.
\end{align}
\vspace{2mm}

Next, we will compute the coefficients of every $h_{44k}^2$ in $d\Phi$, which will be used in some further constructions.

\vspace{2mm}


In $d\theta_{12}$, the term involving $h_{441}^2$ is given by
\begin{align}
    c_{121}h_{441}^2 &= -\frac{h_{331}h_{221}}{(\lambda_3 - \lambda_1) (\lambda_1 - \lambda_2)} - \frac{h_{441}h_{221}}{(\lambda_4 - \lambda_1) (\lambda_1 - \lambda_2)} \nonumber\\[2mm]
    &= \frac{\lambda_2 (\lambda_4 - \lambda_3)}{\lambda_4^2} \cdot \frac{\lambda_4(\lambda_3 - \lambda_2)(\lambda_3 - \lambda_1)^2-\lambda_3(\lambda_4 - \lambda_2)(\lambda_4 - \lambda_1)^2}{(\lambda_3 - \lambda_1)^2 (\lambda_3 - \lambda_2)^2 (\lambda_2 - \lambda_1)^2} h_{441}^2.
\end{align}
Switching indices $2$ and $3$ produces the term in $d\theta_{13}$ involving $h_{441}^2$, denoted by $c_{131}h_{441}^2$, where
\begin{align}
    c_{131}h_{441}^2 &= -\frac{h_{221} h_{331}}{(-\lambda_1 + \lambda_2) (\lambda_1 - \lambda_3)} - \frac{h_{331} h_{441}}{(\lambda_1 - \lambda_3) (-\lambda_1 + \lambda_4)} \nonumber \\[2mm]
    &= \frac{\lambda_3(\lambda_2 - \lambda_4)}{\lambda_4^2} \cdot \frac{\lambda_2(\lambda_4 - \lambda_3)(\lambda_4 - \lambda_1)^2 + \lambda_4(\lambda_3 - \lambda_2)(\lambda_2 - \lambda_1)^2}{(\lambda_3 - \lambda_1)^2 (\lambda_3-\lambda_2)^2 (\lambda_2-\lambda_1)^2} h_{441}^2.
\end{align}

Similarly, interchanging indices $2$ and $4$ leads to the corresponding term in $d\theta_{14}$ involving $h_{441}^2$, denoted by $c_{141}h_{441}^2$. Specifically,
\begin{align}
    c_{141}h_{441}^2 &= -\frac{h_{331}h_{441}}{(\lambda_3-\lambda_1)(\lambda_1-\lambda_4)} - \frac{h_{221}h_{441}}{(\lambda_2-\lambda_1)(\lambda_1-\lambda_4)} \nonumber\\[2mm]
    &= \frac{\lambda_3 - \lambda_2}{\lambda_4} \cdot\frac{\lambda_2(\lambda_4 - \lambda_3)(\lambda_3 - \lambda_1)^2 - \lambda_3 (\lambda_4 - \lambda_2)(\lambda_2 - \lambda_1)^2}{(\lambda_3 - \lambda_1)^2(\lambda_3-\lambda_2)^2(\lambda_2-\lambda_1)^2} h_{441}^2.
\end{align}

\vspace{3mm}

To calculate $c_{231}$, we first extract $c_{123}h_{443}^2$ from \eqref{theta12}:
\begin{align*}
    c_{123}h_{443}^2 &= \frac{h_{113}h_{223}}{(\lambda_1-\lambda_3)(\lambda_3-\lambda_2)}.
\end{align*}
Interchanging indices $1$ and $3$ then gives
\begin{align}
    c_{231}h_{441}^2 &= \frac{h_{331}h_{221}}{(\lambda_3 - \lambda_1)(\lambda_1 - \lambda_2)} = \frac{\lambda_2\lambda_3}{\lambda_4^2}\frac{(\lambda_1-\lambda_4)^2 (\lambda_2-\lambda_4)(\lambda_3-\lambda_4)}{(\lambda_1-\lambda_2)^2 (\lambda_1-\lambda_3)^2 (\lambda_2-\lambda_3)^2} h_{441}^2.
\end{align}

Likewise, for $c_{241}$, we extract the term involving $c_{124}$ from \eqref{theta12} and switch indices $1$ and $4$, resulting in
\begin{equation}
    c_{241} h_{441}^2 = \frac{h_{441}h_{221}}{(\lambda_4-\lambda_1)(\lambda_1-\lambda_2)} = -\frac{\lambda_2}{\lambda_4}\frac{(\lambda_3-\lambda_4)}{(\lambda_1-\lambda_2)^2 (\lambda_2-\lambda_3)} h_{441}^2.
\end{equation}

Finally, to determine $c_{341}$, we extract the term involving $c_{123}$ from \eqref{theta12} and interchange indices $1$ and $3$, followed by $2$ and $4$, which provides
\begin{align}
    c_{341}h_{441}^2 &= \frac{h_{331}^2 h_{441}^2}{(\lambda_3 - \lambda_1)(\lambda_1 - \lambda_4)} = \frac{\lambda_3}{\lambda_4}\frac{(\lambda_2-\lambda_4)}{(\lambda_1-\lambda_3)^2 (\lambda_2-\lambda_3)} h_{441}^2.
\end{align}

\vspace{3mm}

By the symmetry of $\theta_{ij}$, it is sufficient to explicitly compute the coefficient of $h_{441}^2$. The coefficients of the remaining $h_{44k}^2$ terms can subsequently be obtained via index permutation. Based on the preceding calculations, the coefficient of $h_{441}^2$ in $d\Phi$ is
\begin{align}\label{L1}
    L_1 &= c_{121} + c_{131} + c_{141} + c_{231} + c_{241} + c_{341} \nonumber \\[2mm]
    &= -\frac{(\lambda_1+\lambda_2) (\lambda_1+\lambda_3) (\lambda_2+\lambda_3) \Big(\lambda_1^2 (\lambda_2+\lambda_3)+9 \lambda_2\lambda_3 (\lambda_2+\lambda_3) + \lambda_1(\lambda_2^2+10\lambda_2\lambda_3 + \lambda_3^2)\Big)}{(\lambda_1-\lambda_2)^2 (\lambda_1-\lambda_3)^2 (\lambda_2-\lambda_3)^2 \lambda_4^2}.
\end{align}
Similarly, denoting the coefficient of $h_{44k}^2$ in $d\Phi$ by $L_k$, we define
\begin{equation}\label{Lk}
    L_k := -\frac{(\lambda_1+\lambda_2) (\lambda_1+\lambda_3) (\lambda_2+\lambda_3)}{(\lambda_1-\lambda_2)^2 (\lambda_1-\lambda_3)^2 (\lambda_2-\lambda_3)^2 \lambda_4^2}\ C_k, \quad k=2,3,4,
\end{equation}
where each $C_k$ is given by
\begin{align}\label{L234}
    C_2 &= \lambda_2 \lambda_3 (\lambda_2+\lambda_3) + \lambda_1^2 (\lambda_2+9 \lambda_3)+\lambda_1 (\lambda_2^2+10 \lambda_2 \lambda_3+9 \lambda_3^2), \\[2mm]
    C_3 &= \lambda_2 \lambda_3 (\lambda_2+\lambda_3)+\lambda_1^2 (9 \lambda_2+\lambda_3)+\lambda_1 (9 \lambda_2^2+10 \lambda_2 \lambda_3+\lambda_3^2),
\end{align}
\begin{align}
    C_4 &= \frac{\lambda_4^2(\lambda_3 - \lambda_1)^2(\lambda_2 - \lambda_1)^2}{\lambda_1^2 (\lambda_3 - \lambda_4)^2 (\lambda_2 - \lambda_4)^2}\Big(\lambda_1^2 (\lambda_2+\lambda_3)+\lambda_2 \lambda_3 (\lambda_2+\lambda_3)+\lambda_1 (\lambda_2^2-6 \lambda_2 \lambda_3+\lambda_3^2)\Big).
\end{align}

Notice that for the coefficient $C_4$, directly interchanging the indices $1$ and $4$ naturally yields $h_{114}^2$ rather than $h_{444}^2$. To express it in terms of $h_{444}^2$, we apply the relation $h_{114}^2 = \alpha h_{444}^2$ given in \eqref{solu}. This substitution accounts for the additional factor in $C_4$.

\vspace{5mm}

We now introduce a new 3-form $\varphi$ on $M^4$.

\begin{definition}
    By assigning a weight 
    \[
        \mathcal{W}_{ij} = \left(\lambda_i^2 + \lambda_i\lambda_j + \lambda_j^2\right) -\frac{S}{2}, \quad i,j=1,2,3,4
    \]
    to each $\theta_{ij}$, we define the 3-form $\varphi$ on $M^4$ by
    \[
        \varphi := \sum_{i<j} \mathcal{W}_{ij} \theta_{ij}.
    \]
\end{definition}

It then follows that 
\[
    d\varphi = \sum_{i<j} \left( \mathcal{W}_{ij}\ d\theta_{ij} + d\mathcal{W}_{ij}\wedge \theta_{ij} \right).
\]

Because the four principal curvatures are distinct everywhere, for each $p \in M^4$, there exists an open neighborhood $N_p$ of $p$ equipped with a local smooth frame field $\{e_1, e_2, e_3, e_4\}$. With respect to the dual coframe $\{\omega_1, \omega_2, \omega_3, \omega_4\}$, the second fundamental form diagonalizes. That is, 
\[
    \sum_{i,j} h_{ij} \omega_i \otimes \omega_j = \sum_{i,j} \lambda_i \delta_{ij} \omega_i \otimes \omega_j.
\] 
Consequently, differentiating the smooth functions $\lambda_i : N_p \to \mathbb{R}$ along $e_k$ leads to 
\[
    e_k(\lambda_i) = h_{iik} \quad \text{on } N_p.
\]

As an outline of the subsequent argument: for the terms involving $h_{iik}^2$, equation \eqref{solu} implies that it suffices to compute $h_{44k}^2$ using our preliminary results. Furthermore, we will demonstrate that the fully crossing terms $h_{ijk}^2\ (i\ne j \ne k)$ cancel internally. Finally, we will explicitly evaluate the sum of the weighted curvature terms $\mathcal{W}_{ij}R_{ijij}$. Note that all computations hereafter are local.

\subsubsection*{Computation of $d\mathcal{W}_{ij} \wedge \theta_{ij}$} 

By expressing $\theta_{ij} = \omega_k \wedge \omega_l \wedge \omega_{ij}$, we deduce that
\[ d\mathcal{W}_{ij}\wedge\theta_{ij} = \big((\partial_{\lambda_i}\mathcal{W}_{ij})d\lambda_i + (\partial_{\lambda_j}\mathcal{W}_{ij}) d\lambda_j\big)\wedge\omega_k\wedge \omega_l\wedge\omega_{ij},
\]
where $\partial_{\lambda_i}\mathcal{W}_{ij} := 2\lambda_i + \lambda_j$ and $\partial_{\lambda_j}\mathcal{W}_{ij} := \lambda_i + 2\lambda_j$ denote the corresponding partial derivatives. Note that the constant scalar curvature $S$ vanishes upon differentiation.

\vspace{2mm}

 For example, setting $(i,j)=(1,2)$ leads to
\begin{align*}
d\mathcal{W}_{12}\wedge\theta_{12} & = \left(\sum_{p=1}^4 (\partial_{\lambda_1}\mathcal{W}_{12}) h_{11p}\omega_p + (\partial_{\lambda_2}\mathcal{W}_{12}) h_{22p}\omega_p \right)\wedge\omega_3\wedge \omega_4\wedge\left(\sum_{q=1}^4\frac{h_{12q}}{\lambda_{1}-\lambda_2}{\omega_q}\right).
\end{align*}
The only non-vanishing terms occur when $(p,q)=(1,2)$ or $(2,1)$. Consequently, 
\begin{align*}
d\mathcal{W}_{12}\wedge\theta_{12} & = \frac{(\partial_{\lambda_1}\mathcal{W}_{12}) h_{111}h_{122} \!+\! (\partial_{\lambda_2}\mathcal{W}_{12}) h_{221}h_{122}}{\lambda_1 - \lambda_2}\mathrm{vol} - \frac{(\partial_{\lambda_1}\mathcal{W}_{12}) h_{112}h_{121} \!+\! (\partial_{\lambda_2}\mathcal{W}_{12}) h_{222}h_{121}}{\lambda_1 - \lambda_2} \mathrm{vol}.
\end{align*}

Focusing on the terms involving $h_{44k}^2$, we specifically set $k=1$ and extract the corresponding coefficient
\begin{equation}
    w'_{121} = \frac{(\partial_{\lambda_1}\mathcal{W}_{12}) h_{111}h_{221}+ (\partial_{\lambda_2}\mathcal{W}_{12})h_{221}^2}{\lambda_1 - \lambda_2}.
\end{equation}

Similarly, by virtue of symmetry, the corresponding coefficients for the remaining $\theta_{ij}$ are given by
\begin{align}
    w'_{131} &= \frac{(\partial_{\lambda_1}\mathcal{W}_{13}) h_{111}h_{331} + (\partial_{\lambda_3}\mathcal{W}_{13}) h_{331}^2}{\lambda_1 - \lambda_3}, \\[2mm]
    w'_{141} &= \frac{(\partial_{\lambda_1}\mathcal{W}_{14}) h_{111}h_{441}+ (\partial_{\lambda_4}\mathcal{W}_{14}) h_{441}^2}{\lambda_1 - \lambda_4}, \\[2mm]
    w'_{231} &= 0, \quad w'_{241} = 0, \quad w'_{341} = 0,
\end{align}
the last three coefficients vanish due to the presence of repeated basis forms, such as $\omega_1\wedge\omega_1$ or $\omega_2\wedge\omega_2$, in their respective wedge products.

\vspace{2mm}

\subsubsection*{Step 1: the coefficient of $h_{44k}^2$ in $d\varphi$}
Combining the preliminary computations for $d\theta_{ij}$ and $d\mathcal{W}_{ij} \wedge \theta_{ij}$ derived above, we find that the term involving $h_{441}^2$ in $d\varphi$ is 
\begin{align*} 
    \mathcal{L}_1 h_{441}^2 :=\  & (\mathcal{W}_{12}c_{121}h_{441}^2  + w'_{121}) + (\mathcal{W}_{13}c_{131}h_{441}^2 + w'_{131}) + (\mathcal{W}_{14}c_{141} h_{441}^2+ 
    w'_{141}) \\[2mm]
    & + (\mathcal{W}_{23}c_{231}h_{441}^2) + (\mathcal{W}_{24}c_{241}h_{441}^2) + (\mathcal{W}_{34}c_{341}h_{441}^2).
\end{align*}

Substituting \eqref{solu} into $w'_{ij1}$ to eliminate $h_{kk1}$ and simplifying extensively, we find that the coefficient of $h_{441}^2$ in $d\varphi$ is  
\begin{equation}\label{calLk}\mathcal{L}_k=\frac{\lambda_1\lambda_2\lambda_3\mathcal{P}_k(\lambda_1,\lambda_2,\lambda_3)}{4(\lambda_1-\lambda_2)^2 (\lambda_1-\lambda_3)^2 (\lambda_2-\lambda_3)^2 \lambda_4^2},\ \ k=1,2,3,4,\end{equation}
where $\mathcal{P}_k:=\mathcal{P}_k(\lambda_1,\lambda_2,\lambda_3)$ are homogeneous polynomials of degree 5. For example,
\begin{align*}
    \mathcal{P}_1= & \lambda_1^5+5 \lambda_1^4 (\lambda_2+\lambda_3)+(\lambda_2+\lambda_3) (\lambda_2^2+\lambda_2 \lambda_3+\lambda_3^2)^2+\lambda_1^3 (9 \lambda_2^2+20 \lambda_2 \lambda_3+9 \lambda_3^2)\\[2mm]
    &+\lambda_1^2 (7 \lambda_2^3+27 \lambda_2^2 \lambda_3+27 \lambda_2 \lambda_3^2+7 \lambda_3^3)+\lambda_1 (3 \lambda_2^4+14 \lambda_2^3 \lambda_3+23 \lambda_2^2 \lambda_3^2+14 \lambda_2 \lambda_3^3+3 \lambda_3^4),
\end{align*}
and the explicit expressions for the remaining $\mathcal{P}_k$ are provided in \eqref{calP2}, \eqref{calP3}, and \eqref{calP4} of Appendix A.

\vspace{5mm}

According to Lemma \ref{lemma4.3:K>0}, the domain for each $\mathcal{P}_k$ is restricted to $\lambda_1 < \lambda_2 < 0 < \lambda_3 < \lambda_4$. To determine the sign of each $\mathcal{P}_k$, we decompose the domain as follows:
\[
    U_1 := \{\lambda_2+\lambda_3 \le 0\}, \quad U_2 := \{\lambda_2+\lambda_3 > 0 \}.
\]
Within each subregion, we introduce suitable non-negative variables $a, b, c \ge 0$ and examine the signs of the resulting coefficients.

\subsubsection*{The subregion $U_1 =\{\lambda_2+\lambda_3 \le 0\}$:}

On $U_1$, we set $a>0$, $b \ge 0$, and $c>0$ such that
\begin{equation}\label{abcU1}
    \begin{cases}
        \lambda_1 = -a -b -c,\\
        \lambda_2 = -a-b,\\
        \lambda_3 = a.
    \end{cases}
\end{equation}
It then follows that $\lambda_4 = -\lambda_1 - \lambda_2 - \lambda_3 = a + 2b + c > a = \lambda_3$, which guarantees the required ordering $\lambda_1 < \lambda_2 < 0 < \lambda_3 < \lambda_4$. For example, explicitly substituting these relations into $\mathcal{P}_1$, we obtain
\begin{align}
   \mathcal{P}_1 = &\ -26 b^5 - 69 b^4 c - 74 b^3 c^2 - 39 b^2 c^3 - 10 b c^4 - c^5 - 
 4 a^3 (3 b^2 + 3 b c + c^2) \nonumber \\[2mm]
 & -2  a^2  (24 b^3 + 37 b^2 c + 21 b c^2 + 4 c^3) - 
 a (61 b^4 + 128 b^3 c + 105 b^2 c^2 + 38 b c^3 + 5 c^4),
\end{align}

It is evident that all coefficients of the resulting polynomial in $a, b$, and $c$ are negative. Moreover, since certain terms involve only $a$ and $c$, the strict inequality $\mathcal{P}_1 < 0$ remains valid even if $b=0$.

Applying the exact same procedure, we deduce that 
\[
    \mathcal{P}_k < 0, \quad k=2,3,4.
\] 
 For their fully expanded explicit expressions, please see Appendix A.1.

\subsubsection*{The subregion $U_2=\{\lambda_2+\lambda_3 > 0\}$}

Similarly, on $U_2$, we introduce the variables $a > 0$, $b>0$, and $c>0$ defined by
\begin{align}\label{abcU2}
    \begin{cases} 
        \lambda_1 = -a - 2 b - c,\\
        \lambda_2 = -c,\\
        \lambda_3 = b + c.
    \end{cases}
\end{align}
Substituting these expressions into each $\mathcal{P}_k$, we find that the resulting coefficients of $a, b$, and $c$ are all strictly negative. Thus, we have
\[
    \mathcal{P}_k < 0, \quad k=1,2,3,4.
\]
For the completely expanded forms, please see Appendix A.2.

\vspace{2mm}

Therefore, in the expression for $d\varphi$, the coefficients of $h_{44k}^2$ are all strictly negative; that is, 
\[\mathcal{L}_k < 0, \quad k=1,2,3,4.\]

\vspace{3mm}

\subsubsection*{Step 2: fully crossing terms $h_{ijk}^2$} 

We claim that all fully crossing terms $h_{ijk}^2$ in $d\varphi$ cancel out. Indeed, for all distinct $1\le i,j,k \le 4$, the coefficient of $h_{ijk}^2$ in $d\varphi$ satisfies the identity
\begin{align}\label{crossing}
    & \frac{\mathcal{W}_{ij}}{(\lambda_i-\lambda_k)(\lambda_j-\lambda_k)} - \frac{\mathcal{W}_{ik}}{(\lambda_i-\lambda_j)(\lambda_j-\lambda_k)} + \frac{\mathcal{W}_{jk}}{(\lambda_i-\lambda_j)(\lambda_i-\lambda_k)} \nonumber\\[2mm]
    &= \frac{\left(\lambda_i^3-\lambda_j^3 - \lambda_i^3+\lambda_k^3+\lambda_j^3-\lambda_k^3\right) - \frac{S}{2}(\lambda_i-\lambda_j - \lambda_i+\lambda_k+\lambda_j-\lambda_k)}{(\lambda_i-\lambda_k)(\lambda_j-\lambda_k)(\lambda_i-\lambda_j)} \nonumber\\[2mm]
    &= 0,
\end{align}
thus, all the crossing terms vanish in $d\varphi$.


\vspace{2mm}

\subsubsection*{Step 3: weighted curvature terms $\mathcal{W}_{ij}R_{ijij}$} 
Recall that 
\[\sum_{1\le i<j\le 4}R_{ijij} = \frac{1}{2}R_M = \frac{1}{2}\big(n(n-1)-S\big) = \frac{1}{2}(12 - S),\]
and from \eqref{Rijkl}, we have for $i < j$,
\[R_{ijij} = 1 + h_{ii}h_{jj} - h_{ij}^2 = 1 + \lambda_i \lambda_j.\]
A direct calculation shows:
\begin{align*}
    & \quad \sum_{i<j} \left(\lambda_i^2 + \lambda_i\lambda_j + \lambda_j^2 \right) R_{ijij} = \sum_{i<j} \left(\lambda_i^2 + \lambda_i\lambda_j + \lambda_j^2 \right)\big(1 + \lambda_i \lambda_j\big) \\[2mm]
    &= \sum_{i<j} \big(\lambda_i^2 + \lambda_j^2\big) + \sum_{i<j} \lambda_i\lambda_j + \sum_{i<j} \big(\lambda_i^3\lambda_j + \lambda_i\lambda_j^3\big) + \sum_{i<j} \lambda_i^2\lambda_j^2\\[2mm]
    &= 3 \sum_{i} \lambda_i^2 + \frac{\big(\sum_{i} \lambda_i\big)^2-\sum_{i} \lambda_i^2}{2} + \sum_{i \neq j} \lambda_i^3\lambda_j + \frac{\big(\sum_{i} \lambda_i^2\big)^2-\sum_{i} \lambda_i^4}{2} \\[2mm]
    &= 3 \sum_i \lambda_i^2 + \frac{\big(\sum_i \lambda_i\big)^2-\sum_i\lambda_i^2}{2} + \Big(\sum_i\lambda_i^3 \sum_i\lambda_i -\sum_i \lambda_i^4 \Big)+ \frac{\big(\sum_i \lambda_i^2\big)^2-\sum_i \lambda_i^4}{2} \\[2mm]
    &= \frac{5}{2}S + \frac{1}{2}S^2 - \frac{3}{2}f_4.
\end{align*}

\vspace{2mm}

From Lemma \ref{lemmaGB}, we have that the weighted curvature term is 
\begin{align*}
    \mathcal{R} &:= \left(\frac{5}{2}S + \frac{1}{2}S^2 - \frac{3}{2}f_4\right) -\frac{S}{2}\sum_{1\le i<j\le 4}R_{ijij}\\[1.5mm]
    &= \left(-\frac{1}{4}S^2+\frac{7}{2}S-6\right)-\frac{S}{2}\left(\frac{1}{2}(12-S)\right) = \frac{1}{2}S-6.
\end{align*}
Under our assumption that $S>12$, we have $\mathcal{R} > 0$.

\vspace{5mm}

Now, combining the results from \textit{Step 1} through \textit{Step 3}, we establish a contradiction and therefore prove Lemma \ref{lemmaS>12}. 

\begin{proof}[Proof of Lemma \ref{lemmaS>12}]
    Integrating $d\varphi$ over $M^4$ and applying Stokes' Theorem, we have
    \begin{align*}
        0 &= \int_{\partial M^4}\varphi = \int_{M^4} d\varphi = \int_{M^4} \left(\sum_{k=1}^4\mathcal{L}_k h_{44k}^2-\mathcal{R}\right) \mathrm{vol}.
    \end{align*}
    However, as shown in \textit{Step 1}, $\mathcal{L}_k < 0$, and from \textit{Step 3}, $\mathcal{R} > 0$. This implies that the right-hand side is strictly negative, contradicting the left-hand side being $0$.

\vspace{1mm}

    Therefore, we must have $S \le 12$, which is equivalent to the scalar curvature being non-negative.
\end{proof}

\subsection{$S=12$ and $M^4$ is isoparametric}

\begin{lemma}\label{lemmaS=12}
Let $M^4 \hookrightarrow \mathbb{S}^5(1)$ be a closed minimal hypersurface with constant scalar curvature, constant Gauss-Kronecker curvature, and four distinct principal curvatures everywhere. If the scalar curvature is non-negative (i.e. $S \le 12$), then $S=12$ and $M^4$ is isoparametric.
\end{lemma}

To begin with, we slightly modify the 3-form $\varphi$ to construct a new 3-form $\Psi$ for this subsection.

\begin{definition}
    By assigning the new weight 
    \[
        W_{ij} = \lambda_i^2 + \lambda_i\lambda_j + \lambda_j^2 = \mathcal{W}_{ij} + \frac{S}{2}, \quad i,j=1,2,3,4
    \]
    to each $\theta_{ij}$, we define a different 3-form $\Psi$ on $M^4$ by 
    \[
        \Psi = \sum_{i<j}W_{ij}\theta_{ij} = \varphi + \frac{S}{2}\Phi,
    \]
    where $\theta_{ij}$ and $\Phi$ are given in Definition \ref{theta:Phi}.
\end{definition}

Similarly, we will compute $d\Psi$ using the results established in the previous subsection. Note that all subsequent computations are local.

\subsubsection*{Step 1: the coefficient of $h_{44k}^2$ in $d\Psi$}

Thanks to $d\mathcal{W}_{ij} = dW_{ij}$, substituting \eqref{L1}, \eqref{Lk}, \eqref{L234}, and \eqref{calLk} into the exterior derivative $d\Psi = d\varphi + \frac{S}{2}d\Phi$, we find that the coefficient of $h_{441}^2$ in $d\Psi$ is given by
\begin{equation}\label{tilL}
    \tilde{L}_1 =\mathcal{L}_1 + \frac{S}{2}L_1 =-\frac{P_1(\lambda_1,\lambda_2,\lambda_3)}{(\lambda_1-\lambda_2)^2 (\lambda_1-\lambda_3)^2 (\lambda_2-\lambda_3)^2 \lambda_4^2},
\end{equation}
where the numerator $P_1$ is 
\begin{align}\label{tilL1}
P_1(\lambda_1,\lambda_2,\lambda_3)  = & \lambda_1^6 \big(\lambda_2-\lambda_3\big)^2+3 \lambda_1^5 \big(\lambda_2-\lambda_3\big)^2 \big(\lambda_2+\lambda_3\big)+9 \lambda_2^2 \lambda_3^2 \big(\lambda_2+\lambda_3\big)^2 \big(\lambda_2^2+\lambda_2 \lambda_3+\lambda_3^2\big) \nonumber \\[2mm]
    &+\lambda_1^4 \big(4 \lambda_2^4+5 \lambda_2^3 \lambda_3-6\lambda_2^2\lambda_3^2+5 \lambda_2 \lambda_3^3+4 \lambda_3^4\big) \nonumber\\[2mm]
    & +3 \lambda_1^3 \big(\lambda_2^5+7 \lambda_2^4 \lambda_3+8 \lambda_2^3 \lambda_3^2+8 \lambda_2^2 \lambda_3^3 +7 \lambda_2 \lambda_3^4+\lambda_3^5\big) \nonumber\\[2mm]
    & +3 \lambda_1 \lambda_2 \lambda_3 \big(2 \lambda_2^5+15 \lambda_2^4 \lambda_3+31 \lambda_2^3 \lambda_3^2+31 \lambda_2^2 \lambda_3^3+15 \lambda_2 \lambda_3^4+2 \lambda_3^5\big) \nonumber\\[2mm]
    & +\lambda_1^2 \big(\lambda_2^6+21 \lambda_2^5 \lambda_3+66 \lambda_2^4 \lambda_3^2+88 \lambda_2^3 \lambda_3^3+66 \lambda_2^2 \lambda_3^4+21 \lambda_2 \lambda_3^5+\lambda_3^6\big).
\end{align}
Similarly, denoting the coefficient of $h_{44k}^2$ in $d\Psi$ by $\tilde{L}_k$, we define
\begin{equation}\label{tilLk}
    \tilde{L}_k := -\frac{P_k(\lambda_1,\lambda_2,\lambda_3)}{(\lambda_1-\lambda_2)^2 (\lambda_1-\lambda_3)^2 (\lambda_2-\lambda_3)^2 \lambda_4^2}, \quad k=2,3,4,
\end{equation}
where each $P_k := P_k(\lambda_1, \lambda_2, \lambda_3)$ is a homogeneous polynomial of degree $8$. Please see \eqref{tilL2}, \eqref{tilL3}, and \eqref{tilL4} in Appendix B for their full explicit expressions.

\vspace{3mm}

According to Lemma \ref{lemma4.3:K>0}, the domain for each $P_k$ is restricted to $\lambda_1 < \lambda_2 < 0 < \lambda_3 < \lambda_4$. Following a similar argument, to determine the sign of each $P_k$, we partition the domain into four subregions:
\[
    V_1 := \{-\lambda_2 \ge 3\lambda_3\}, \ 
    V_2 := \{\lambda_3 \le -\lambda_2 \le 3\lambda_3\}, \ 
    V_3 := \left\{\frac{1}{3}\lambda_3 \ge -\lambda_2 \right\}, \ 
    V_4 := \left\{\frac{1}{3}\lambda_3 < -\lambda_2 < \lambda_3 \right\},
\]
and within each subregion, we introduce a suitable change of variables involving non-negative parameters $a, b, c \ge 0$.

\subsubsection*{The subregion $V_1$}

On $V_1 =\{-\lambda_2 \ge 3\lambda_3\}$, we set $a>0$, $b\ge 0$, and $c>0$ such that
\begin{equation}\label{abc}
    \begin{cases}
        \lambda_1 = -3a - b - c, \\
        \lambda_2 = -3a - b,\\
        \lambda_3 = a.
    \end{cases}
\end{equation}
It then follows that $\lambda_4 = -\lambda_1 - \lambda_2 - \lambda_3 = 5a+2b+c > a = \lambda_3$. It is evident that $\lambda_1 < \lambda_2 < 0 < \lambda_3$ and $-\lambda_2 \ge 3\lambda_3$, which ensures that this change of variables satisfies all constraints on the $\lambda_i$ within $V_1$.

\vspace{1mm}

Substituting \eqref{abc} into \eqref{tilL1}, expanding completely, and collecting terms with respect to descending powers of $a$, we obtain the following expression:
\begin{align*}
    P_1 =\ & 29952 a^8 + 960 a^7 (91 b + 61 c)  + 
 16 a^6 (6967 b^2 + 9239 b c + 2929 c^2) \\[1.5mm]
 & + 
 24 a^5 (3400 b^3 + 6635 b^2 c + 4069 b c^2 + 762 c^3) \\[1.5mm]
 &  + 
 4 a^4 (9415 b^4 + 23788 b^3 c + 21069 b^2 c^2 + 7600 b c^3 + 
    931 c^4) \\[1.5mm]
 & + 
 8 a^3 (1410 b^5 + 4273 b^4 c + 4828 b^3 c^2 + 2504 b^2 c^3+ 
    587 b c^4 + 48 c^5)\\[1.5mm]
 & + 
 2 a^2 (b + c)^2 (1076 b^4 + 1556 b^3 c + 783 b^2 c^2 + 152 b c^3 + 
    8 c^4) \\[1.5mm]
&  + 
 8 a b (b + c)^2 (30 b^4 + 53 b^3 c + 35 b^2 c^2 + 10 b c^3 + c^4) \\[1.5mm]
 & + b^2 (2 b^2 + 3 b c + c^2)^2 (3 b^2 + 3 b c + c^2).
\end{align*}
Since $a>0,\ b\ge0,\ c > 0$ and all the resulting coefficients are strictly positive, $P_1 > 0$.

\vspace{2mm}
Similarly, substituting \eqref{abc} into $P_2, P_3$, and $P_4$, we obtain polynomials in $a, b$, and $c$ whose coefficients are all strictly positive. Consequently, we deduce that
\[P_k > 0, \quad k = 2, 3, 4.\]
For their explicit expressions, please see Appendix B.1.

\vspace{2mm}
 
 \subsubsection*{The subregion $V_2$}
Similar to our approach for $V_1$, on $V_2  = \{\lambda_3 \le - \lambda_2 \le 3\lambda_3\}$, we introduce the variables $a\ge 0$, $b\ge 0$, and $c>0$ defined by
\begin{equation}\label{abcV2}
    \begin{cases}
        \lambda_1 = -3 a - b - c,\\
        \lambda_2 = -3a - b,\\
        \lambda_3 = a + b.
    \end{cases}
\end{equation}
Note that $a$ and $b$ cannot vanish simultaneously; otherwise, we would have $\lambda_2=\lambda_3=0$, which is a contradiction. It is straightforward to verify that $\lambda_4 = 5a + b + c > \lambda_3$. Furthermore, we have $\lambda_1 < \lambda_2 < 0 < \lambda_3$ and $\lambda_3 \le -\lambda_2 \le 3\lambda_3$, confirming that this substitution is valid within $V_2$.

\vspace{2mm}

Thus, by substituting \eqref{abcV2} into \eqref{tilL1} and \eqref{tilLk} and fully expanding the results, we obtain polynomials $P_k$ in $a, b$, and $c$ with strictly positive coefficients for $k=1,2,3,4$. Because certain terms involve only $a$ and $c$, or only $b$ and $c$, we have the strict inequalities
\[
    P_k > 0, \quad k=1,2,3,4.
\]
For the completely expanded explicit expressions, please see Appendix B.2.

\subsubsection*{The subregion $V_3$}

In $V_3 = \left\{\frac{1}{3}\lambda_3 \ge -\lambda_2 \right\}$, we introduce the variables $a>0$, $b\ge 0$, and $c>0$ such that
\begin{equation}\label{abcV3}
    \begin{cases}
        \lambda_1 = -5 a - 2 b - c,\\
        \lambda_2 = -a, \\
        \lambda_3 = 3 a + b.
    \end{cases}
\end{equation}
A direct verification shows that $\lambda_4 = 3a + b + c > \lambda_3 > 0 > \lambda_2 > \lambda_1$.

Similarly, upon substituting \eqref{abcV3} into \eqref{tilL1} and \eqref{tilLk}, we find that all coefficients in the resulting polynomials in $a, b$, and $c$ are positive. Because there are terms involving only $a$ and $c$, the strict inequality holds even if $b=0$. We therefore have
\[
    P_k > 0, \quad k=1,2,3,4.
\]
For the complete explicit expressions, please see Appendix B.3.

\subsubsection*{The subregion $V_4$}

Finally, on $V_4  = \left\{\frac{1}{3}\lambda_3 < -\lambda_2 < \lambda_3 \right\}$, we set $a>0$, $b>0$, and $c>0$ via the relations
\begin{equation}\label{abcV4}
    \begin{cases}
        \lambda_1 = -5 a - b - c,\\
        \lambda_2 = -a - b,\\
        \lambda_3 = 3 a + b,
    \end{cases}
\end{equation}
from which we deduce $\lambda_4 = 3a + b + c > \lambda_3$.

Again, after substitution and expansion, we observe that the polynomials in $a, b$, and $c$ possess strictly positive coefficients. Since $a, b, c > 0$, we conclude that
\[
    P_k > 0, \quad k=1,2,3,4.
\]
For their detailed expressions, please see Appendix B.4.

\vspace{3mm}

Therefore, from \eqref{tilL1} and \eqref{tilLk}, we have $P_k > 0$ within each subregion, and consequently throughout the entire domain $\lambda_1 < \lambda_2 < 0 < \lambda_3 < \lambda_4$. It follows that the coefficients of $h_{44k}^2$ in $d\Psi$ are all strictly negative; that is,
\[ \tilde{L}_k < 0, \quad k=1,2,3,4.\]

\vspace{2mm}

\subsubsection*{Step 2: fully crossing terms $h_{ijk}^2$} 

Because $\Psi$ is a linear combination of $\varphi$ and $\Phi$, and we have established in \eqref{crossing} and earlier computations that neither $d\varphi$ nor $d\Phi$ contains fully crossing terms, we deduce that all fully crossing terms $h_{ijk}^2$ (with distinct $1\le i,j,k\le 4$) vanish in $d\Psi$ as well.

\vspace{2mm}

\subsubsection*{Step 3: weighted curvature terms $W_{ij}R_{ijij}$}

Recall that $\sum_{i<j} R_{ijij} = \frac{1}{2}R_M = \frac{1}{2}(12-S)$, so we obtain the sum of the weighted curvature terms as
\begin{equation}\label{tilR}
    \tilde{R} = \mathcal{R} + \frac{S}{2}\left(\frac{1}{2}(12-S)\right) = -\frac{1}{4}S^2+\frac{7}{2}S-6 = -\frac{1}{4}(S-12)(S-2).
\end{equation}
Consequently, when $S \le 12$, noting that we also have $S > 6$ by Lemma \ref{lemmaGB}, we conclude that $\tilde{R} \ge 0$.

\vspace{3mm}

We are now in a position to prove Lemma \ref{lemmaS=12}.
\begin{proof}[Proof of Lemma \ref{lemmaS=12}]
    Integrating $d\Psi$ over $M^4$ and applying Stokes' Theorem, we have
    \[ 0 = \int_{\partial M^4}\Psi = \int_{M^4} d\Psi = \int_{M^4} \left(\sum_{k=1}^4\tilde{L}_k h_{44k}^2 - \tilde{R}\right) \mathrm{vol}. \]
    Since $\tilde{L}_k < 0$ and $\tilde{R} \ge 0$, the only way to avoid a contradiction is if $h_{44k} \equiv 0$ and $\tilde{R} = 0$ everywhere on $M^4$. Therefore, by \eqref{solu}, $h_{44k} \equiv 0$ forces $h_{iik} \equiv 0$ for all $i,k=1,2,3,4$. Furthermore, the condition $\tilde{R}=0$ implies $S=12$.  

    In conclusion, we have $S=12$ and $M^4$ is isoparametric.
\end{proof}

\vspace{2mm}

In summary, we present the proof of Theorem \ref{th4}.

\begin{proof}[Proof of Theorem \ref{th4}]
    First, by Lemma \ref{lemma4.3:K>0}, the constant Gauss-Kronecker curvature must be positive, which implies that the principal curvatures satisfy $\lambda_1 < \lambda_2 < 0 < \lambda_3 < \lambda_4$.

\vspace{1mm}

    Next, Lemma \ref{lemmaS>12} ensures that $M^4$ has non-negative scalar curvature; that is, $S \le 12$.

\vspace{1mm}

    Finally, Lemma \ref{lemmaS=12} demonstrates that we actually have $S=12$ and $M^4$ is isoparametric. According to the classification theorem in \cite{C}, $M^4$ must be a Cartan minimal hypersurface.
\end{proof}


\vspace{5mm}


\section{Three distinct principal curvatures at one point}
In this section, we demonstrate that $M^4$ cannot have exactly three distinct principal curvatures at any point.

Suppose there exists a point $p \in M^4$ at which $M^4$ has exactly three distinct principal curvatures. Ordering the principal curvatures from least to greatest, the following two cases cover all possible configurations at $p$:
\[
    \text{Case I: } \lambda_1 < \lambda_2 = \lambda_3 < \lambda_4; \quad \text{Case II: } \lambda_1 = \lambda_2 < \lambda_3 < \lambda_4 \text{ or } \lambda_1 < \lambda_2 < \lambda_3 = \lambda_4.
\]
We examine these two cases separately.

\addtocontents{toc}{\protect\setcounter{tocdepth}{-1}} 
\subsection*{Case I. $\lambda_2=\lambda_3$}

In fact, we establish a stronger result for this scenario, as stated below.

\begin{lemma}\label{3:Kfeifu}
    Let $M^4 \hookrightarrow \mathbb{S}^5(1)$ be a closed minimal hypersurface with constant scalar curvature and constant Gauss-Kronecker curvature. If there exists a point $p \in M^4$ at which $M^4$ has exactly 3 distinct principal curvatures, then the Gauss-Kronecker curvature must be positive.
\end{lemma}
\begin{proof}[Proof of Lemma \ref{3:Kfeifu}]
    We proceed by contradiction. 
    Suppose that the constant Gauss-Kronecker curvature is not positive, i.e. $\mathcal{K} \le 0$. We claim that $\lambda_4 > 0$ and thus $\lambda_1 < 0$. 
    
    Indeed, if the greatest principal curvature $\lambda_4 \le 0$, the minimality of $M^4$ forces $\lambda_1 = \lambda_2 = \lambda_3 = \lambda_4 = 0$. Consequently, $M^4$ would be totally geodesic, which contradicts the existence of three distinct principal curvatures at $p$. 

\vspace{1.5mm}

    Since $\mathcal{K} \le 0$, it algebraically follows that the product $\lambda_2 \lambda_3 \ge 0$. WLOG, we may assume that $\lambda_1 < 0$, $\lambda_2, \lambda_3 \le 0$, and $\lambda_4 > 0$.
    \vspace{1.5mm}

    Continuity ensures that the signs of the principal curvatures remain constant across $M^4$. 
    Because $\lambda_4$ is strictly positive while all other principal curvatures are non-positive, $\lambda_4$ must be a simple eigenvalue of the second fundamental form $h$ across the entire $M^4$. Therefore, by Lemma \ref{lemmaGB}, we obtain $\chi(M^4) = 0$, and thus
    \[S = 6(\mathcal{K} + 1) \le 6.\]
    
    From \eqref{eq:b1}, we have $S \ge 4$. However, $S$ cannot equal $4$ because $M^4$ possesses exactly 3 distinct principal curvatures at $p$, which means it cannot be a Clifford torus. We thus have $S > 4$. Applying Theorem \ref{YC3} (Yang-Cheng \cite{YC3}) then provides
    \[
        S > n + \frac{2}{3}n = \frac{20}{3}.
    \]
    This contradicts $S \le 6$, completing the proof.
\end{proof}

\begin{remark} 
    Under the assumption that $\lambda_2 = \lambda_3$ at $p$, we must have either $\lambda_1 < \lambda_2 = \lambda_3 \le 0 < \lambda_4$ or $\lambda_1 < 0 \le \lambda_2 = \lambda_3 < \lambda_4$. In either scenario, the Gauss-Kronecker curvature $\mathcal{K} = \lambda_1\lambda_2\lambda_3\lambda_4\le 0$. Consequently, Lemma \ref{3:Kfeifu} shows that this case is impossible. 
\end{remark}

\vspace{3mm}

\subsection*{Case II. $\lambda_1=\lambda_2$ or $\lambda_3=\lambda_4$}
To simplify our argument in this case, we first rule out the scenarios have already been solved.

Since we have already addressed the cases where $M^4$ contains a point with two distinct principal curvatures, we now focus on the situation where every point on $M^4$ has exactly three or four distinct principal curvatures. Note that this inherently excludes the possibility of having $\lambda_1 = \lambda_2 < 0 < \lambda_3 = \lambda_4$ at any single point.

\vspace{2mm}

Furthermore, Lemma \ref{3:Kfeifu} establishes that the constant Gauss-Kronecker curvature must be strictly positive. Thus, the cases  $\lambda_1 = \lambda_2 < \lambda_3 < 0 < \lambda_4$ and $\lambda_1 < 0 < \lambda_2 < \lambda_3 = \lambda_4$ are impossible, because they would result in a strictly negative Gauss-Kronecker curvature. With these scenarios excluded, it automatically follows that the principal curvatures in our current case must satisfy
\[
    \lambda_1 \le \lambda_2 < 0 < \lambda_3 < \lambda_4 \ \text{ or } \  \lambda_1 < \lambda_2 < 0 < \lambda_3 \le \lambda_4.
\]

Therefore, in this subsection, we focus on the specific subset of points where exactly three distinct principal curvatures occur, that is:
\[
    P \sqcup Q := \{p \in M^4: \lambda_1(p) = \lambda_2(p)\} \sqcup \{q \in M^4: \lambda_3(q) = \lambda_4(q)\}.
\]
As we assume that there exists a point $p\in M^4$ at which $M^4$ has three distinct principal curvatures, then $P \sqcup Q \neq \emptyset$. 

To simplify notation, we will use $\{f=c\}$ to denote the level set $\{x \in M^4: f(x)=c\}$. For instance, $P=\{\lambda_1 = \lambda_2\}$ and $Q = \{\lambda_3 = \lambda_4\}$.


\vspace{3mm}

Inspired by \cite{SX}, we general their cut-off technique to our case based on the results in Section 4. We will first prove that the scalar curvature is non-negative, meaning $S \le 12$. Subsequently, this condition forces $S = 12$, which creates a geometrical contradiction and shows that this geometric case cannot happen.



\begin{lemma}\label{3roots:S>12}
    Let $M^4 \hookrightarrow \mathbb{S}^5(1)$ be a closed minimal hypersurface with constant scalar curvature and constant Gauss-Kronecker curvature. If there exists a point $p \in M^4$ at which there are exactly three distinct principal curvatures, then $S \le 12$.
\end{lemma}
Since the scalar curvature of $M^4$ is given by $R_M=12-S$, Lemma \ref{3roots:S>12} implies that the scalar curvature of $M^4$ must be non-negative.
 
\vspace{1mm}

To show this using a cut-off argument, we first construct new 3-forms based on $\varphi$ and establish the following estimates for them:
\begin{lemma}\label{lemma:dfdg}
    Let $g = (\lambda_2 - \lambda_1)^2$ and $f = (\lambda_4 - \lambda_3)^2$. There exist constants $\delta > 0$ and $C > 0$ such that:
    \begin{itemize}
        \item[(i)] on the set $\{0 < g < \delta\}$, we have 
        $\displaystyle{ dg \wedge \varphi \le C \left(\sum_{k=1}^4 h_{44k}^2\right) \mathrm{vol}};$\\[2mm]
        \item[(ii)] on the set $\{0 < f < \delta\}$, we have  
        $\displaystyle{df \wedge \varphi \le C \left(\sum_{k=1}^4  (\lambda_4 - \lambda_3)^2 h_{44k}^2\right) \mathrm{vol}.}$
    \end{itemize}
\end{lemma}
\begin{proof}[Proof of Lemma \ref{lemma:dfdg}]
\textbf{Step 1: (calculate $dg\wedge\varphi$). } We will provide a detailed demonstration of $dg \wedge \varphi$, then $ df \wedge \varphi$ is similar. Because \[dg \wedge \varphi = \sum_{i=1}^4 g_i \omega_i \wedge \varphi,\]
We begin by computing $\omega_i\wedge \varphi$ and the first derivatives $g_i$ separately.
\begin{align*}
\omega_1 \wedge \varphi &= \omega_1 \wedge (\theta_{12} + \theta_{13} + \theta_{14}) \nonumber \\[2mm]
&= \left[ -\mathcal{W}_{12}\left(\frac{\lambda_2}{\lambda_4}\right)\frac{(\lambda_4 - \lambda_3)(\lambda_4 - \lambda_1)}{(\lambda_3 - \lambda_2)(\lambda_2 - \lambda_1)^2} + \mathcal{W}_{13}\left(\frac{\lambda_3}{\lambda_4}\right)\frac{(\lambda_4 - \lambda_2)(\lambda_4 - \lambda_1)}{(\lambda_3 - \lambda_2)(\lambda_3 - \lambda_1)^2} - \frac{\mathcal{W}_{14}}{\lambda_4 - \lambda_1} \right] h_{441} \mathrm{vol},
\end{align*}
\begin{align*}
\omega_2 \wedge \varphi &= \omega_2 \wedge (\theta_{12} + \theta_{23} + \theta_{24}) \nonumber \\[2mm]
&= \left[ \mathcal{W}_{23} \left( \frac{\lambda_3}{\lambda_4}\right)\frac{(\lambda_4 - \lambda_2)(\lambda_4 - \lambda_1)}{(\lambda_3 - \lambda_1)(\lambda_3 - \lambda_2)^2} - \frac{\mathcal{W}_{24}}{\lambda_4 - \lambda_2} - \mathcal{W}_{14}\left(\frac{\lambda_1}{\lambda_4}\right)\frac{(\lambda_4 - \lambda_3)(\lambda_4 - \lambda_2)}{(\lambda_3 - \lambda_1)(\lambda_2 - \lambda_1)^2} \right] h_{442} \mathrm{vol}, \end{align*}
\begin{align*}
\omega_3 \wedge \varphi &= \omega_3 \wedge (\theta_{13} + \theta_{23} + \theta_{34}) \nonumber \\[2mm]
&= \left[ - \frac{\mathcal{W}_{34}}{\lambda_4 - \lambda_3} + \mathcal{W}_{23}\left(\frac{\lambda_2}{\lambda_4}\right)\frac{(\lambda_4 - \lambda_3)(\lambda_4 - \lambda_1)}{(\lambda_3 - \lambda_2)^2(\lambda_2 - \lambda_1)} - \mathcal{W}_{13}\left(\frac{\lambda_1}{\lambda_4}\right)\frac{(\lambda_4 - \lambda_3)(\lambda_4 - \lambda_2)}{(\lambda_3 - \lambda_1)^2(\lambda_2 - \lambda_1)} \right] h_{443} \mathrm{vol}, \end{align*}
\begin{align*}
\omega_4 \wedge \varphi &= \omega_4 \wedge (\theta_{14} + \theta_{24} + \theta_{34}) \nonumber \\[2mm]
&= \left[ - \mathcal{W}_{14}\left(\frac{\lambda_1}{\lambda_4}\right)\frac{(\lambda_4 - \lambda_3)(\lambda_4 - \lambda_2)}{(\lambda_3 - \lambda_1)(\lambda_2 - \lambda_1)(\lambda_4 - \lambda_1)} + \mathcal{W}_{24}\left(\frac{\lambda_2}{\lambda_4}\right)\frac{(\lambda_4 - \lambda_3)(\lambda_4 - \lambda_1)}{(\lambda_2 - \lambda_1)(\lambda_3 - \lambda_2)(\lambda_4 - \lambda_2)} \right. \\[2mm]
&\quad \left. - \mathcal{W}_{34} \left(\frac{\lambda_3}{\lambda_4}\right)\frac{(\lambda_4 - \lambda_2)(\lambda_4 - \lambda_1)}{(\lambda_3 - \lambda_1)(\lambda_3 - \lambda_2)(\lambda_4 - \lambda_3)} \right] h_{444} \mathrm{vol}. \nonumber
\end{align*}

Taking the first derivative of $g=(\lambda_2-\lambda_1)^2$, we obtain
\begin{align*}
    g_i&=2(\lambda_2 - \lambda_1)(h_{22i} - h_{11i})\\[2mm]
    &= 2(\lambda_2 - \lambda_1) \left[ \frac{\lambda_2}{\lambda_4}\frac{(\lambda_4 - \lambda_3)(\lambda_4 - \lambda_1)}{(\lambda_2 - \lambda_1)(\lambda_3 - \lambda_2)} - \left( -\frac{\lambda_1}{\lambda_4}\frac{(\lambda_4 - \lambda_3)(\lambda_4 - \lambda_2)}{(\lambda_3 - \lambda_1)(\lambda_2 - \lambda_1)} \right) \right] h_{44i} \nonumber \\[2mm]
    &= \frac{2(\lambda_4 - \lambda_3)}{\lambda_4} \left( \frac{\lambda_2(\lambda_4 - \lambda_1)}{\lambda_3 - \lambda_2} + \frac{\lambda_1(\lambda_4 - \lambda_2)}{\lambda_3 - \lambda_1} \right) h_{44i} =: \tilde{m}_g h_{44i}
\end{align*}
Substituting into $dg\wedge \varphi$, we have 
\begin{align*}
dg \wedge \varphi =& \sum_{i=1}^4 g_i \omega_i \wedge \varphi \nonumber \\[2mm]
=&\  \tilde{m}_g \left[ - \mathcal{W}_{12}\frac{\lambda_2}{\lambda_4}\frac{(\lambda_4 - \lambda_3)(\lambda_4 - \lambda_1)}{(\lambda_3 - \lambda_2)(\lambda_2 - \lambda_1)^2} + \mathcal{W}_{13}\frac{\lambda_3}{\lambda_4}\frac{(\lambda_4 - \lambda_2)(\lambda_4 - \lambda_1)}{(\lambda_3 - \lambda_2)(\lambda_3 - \lambda_1)^2} - \mathcal{W}_{14}\frac{1}{\lambda_4 - \lambda_1} \right] h_{441}^2 \mathrm{vol} \\[2mm]
& + \tilde{m}_g \left[ \mathcal{W}_{23}\frac{\lambda_3}{\lambda_4}\frac{(\lambda_4 - \lambda_2)(\lambda_4 - \lambda_1)}{(\lambda_3 - \lambda_1)(\lambda_3 - \lambda_2)^2} - \mathcal{W}_{24}\frac{1}{\lambda_4 - \lambda_2} - \mathcal{W}_{12}\frac{\lambda_1}{\lambda_4}\frac{(\lambda_4 - \lambda_3)(\lambda_4 - \lambda_2)}{(\lambda_3 - \lambda_1)(\lambda_2 - \lambda_1)^2} \right] h_{442}^2 \mathrm{vol} \\[2mm]
& + \tilde{m}_g \left[ - \mathcal{W}_{34}\frac{1}{\lambda_4 - \lambda_3} + \mathcal{W}_{23}\frac{\lambda_2}{\lambda_4}\frac{(\lambda_4 - \lambda_3)(\lambda_4 - \lambda_1)}{(\lambda_3 - \lambda_2)^2(\lambda_2 - \lambda_1)} - \mathcal{W}_{13}\frac{\lambda_1}{\lambda_4}\frac{(\lambda_4 - \lambda_3)(\lambda_4 - \lambda_2)}{(\lambda_3 - \lambda_1)^2(\lambda_2 - \lambda_1)} \right] h_{443}^2 \mathrm{vol} \\[2mm]
&  + \tilde{m}_g \left[ - \mathcal{W}_{14}\frac{\lambda_1}{\lambda_4}\frac{(\lambda_4 - \lambda_3)(\lambda_4 - \lambda_2)}{(\lambda_3 - \lambda_1)(\lambda_2 - \lambda_1)(\lambda_4 - \lambda_1)} + \mathcal{W}_{24}\frac{\lambda_2}{\lambda_4}\frac{(\lambda_4 - \lambda_3)(\lambda_4 - \lambda_1)}{(\lambda_2 - \lambda_1)(\lambda_3 - \lambda_2)(\lambda_4 - \lambda_2)} \right.  \\[2mm]
&\ \ \ \ \ \ \ \ \  \left. - \mathcal{W}_{34}\frac{\lambda_3}{\lambda_4}\frac{(\lambda_4 - \lambda_2)(\lambda_4 - \lambda_1)}{(\lambda_3 - \lambda_1)(\lambda_3 - \lambda_2)(\lambda_4 - \lambda_3)} \right] h_{444}^2 \mathrm{vol}.
\end{align*}

\textbf{Step 2: (potential singularities). } 
Among the coefficients containing $(\lambda_2 - \lambda_1)$ in the denominator, we identify two "good" terms and two "bad" terms. A term is classified as "good" if the singularity $(\lambda_2 - \lambda_1)$ in the denominator is canceled by a corresponding factor in the numerator. Conversely, a term is "bad" if it remains genuinely singular and may diverge as $\lambda_1 \to \lambda_2$.

\vspace{1mm}

Recall that the weight $\mathcal{W}_{ij}= (\lambda_i^2 + \lambda_i\lambda_j + \lambda_j^2) - \frac{S}{2} =W_{ij} - \frac{S}{2}$.

\vspace{2mm}

\textbf{Step 2.1: (the good terms). } 

The first good term from $h_{443}^2$ is given by
\begin{align}\label{good1}
 & \tilde{m}_g \left[\frac{\mathcal{W}_{23}\lambda_2(\lambda_4 - \lambda_3)(\lambda_4 - \lambda_1)}{\lambda_4(\lambda_3 - \lambda_2)^2(\lambda_2 - \lambda_1)} - \frac{\mathcal{W}_{13}\lambda_1(\lambda_4 - \lambda_3)(\lambda_4 - \lambda_2)}{\lambda_4(\lambda_3 - \lambda_1)^2(\lambda_2 - \lambda_1)} \nonumber\right] \\[2mm]
    = \ & \frac{\tilde{m}_g (\lambda_4 - \lambda_3)}{\lambda_4(\lambda_2 - \lambda_1)} \left( \frac{\mathcal{W}_{23}\lambda_2(\lambda_4 - \lambda_1)}{(\lambda_3 - \lambda_2)^2} - \frac{\mathcal{W}_{13}\lambda_1(\lambda_4 - \lambda_2)}{(\lambda_3 - \lambda_1)^2} \right) =: \frac{\mathcal{G}_1(\lambda_1, \lambda_2, \lambda_3, \lambda_4)}{\lambda_2 - \lambda_1},
\end{align}
where $\mathcal{G}_1(\lambda_1, \lambda_2, \lambda_3, \lambda_4)$ is a smooth rational function of $\lambda_i$. Notice that substituting $\lambda_2 = \lambda_1$ gives $\mathcal{G}_1(\lambda_1, \lambda_1, \lambda_3, \lambda_4) = 0$. 

Therefore, as $\lambda_2 \to \lambda_1$, we recognize this expression as a difference quotient converging to the partial derivative with respect to its second component
\[
    \lim_{\lambda_2 \to \lambda_1} \frac{\mathcal{G}_1(\lambda_1, \lambda_2, \lambda_3, \lambda_4) }{\lambda_2 - \lambda_1} = \lim_{\lambda_2 \to \lambda_1} \frac{\mathcal{G}_1(\lambda_1, \lambda_2, \lambda_3, \lambda_4) - \mathcal{G}_1(\lambda_1, \lambda_1, \lambda_3, \lambda_4)}{\lambda_2 - \lambda_1} = \partial_{\lambda_2} \mathcal{G}_1(\lambda_1, \lambda_1, \lambda_3, \lambda_4).
\]
Since each $\lambda_i$ is bounded on $M^4$, this limit is uniformly bounded. Thus, the apparent singularity at $\lambda_2 = \lambda_1$ is removable, and there exists a $\delta_1 > 0$ such that the first good term is uniformly bounded on the set $\{0 < g < \delta_1\}$.

\vspace{2mm}

Similarly, the second good term from $h_{444}^2$ can be expressed as
\begin{align}\label{good2}
     & \tilde{m}_g \left[- \frac{\mathcal{W}_{14}\lambda_1(\lambda_4 - \lambda_3)(\lambda_4 - \lambda_2)}{\lambda_4(\lambda_3 - \lambda_1)(\lambda_2 - \lambda_1)(\lambda_4 - \lambda_1)} + \frac{\mathcal{W}_{24}\lambda_2(\lambda_4 - \lambda_3)(\lambda_4 - \lambda_1)}{\lambda_4(\lambda_2 - \lambda_1)(\lambda_3 - \lambda_2)(\lambda_4 - \lambda_2)}\right] \nonumber \\[2mm]
= \ & \frac{\lambda_4 - \lambda_3}{\lambda_4(\lambda_2 - \lambda_1)} \cdot \left( \frac{\mathcal{W}_{24}\lambda_2(\lambda_4 - \lambda_1)}{(\lambda_3 - \lambda_2)(\lambda_4 - \lambda_2)} - \frac{\mathcal{W}_{14}\lambda_1(\lambda_4 - \lambda_2)}{(\lambda_3 - \lambda_1)(\lambda_4 - \lambda_1)} \right) =: \frac{\mathcal{G}_2(\lambda_1 , \lambda_2, \lambda_3, \lambda_4)}{\lambda_2 - \lambda_1},
\end{align}
where $\mathcal{G}_2(\lambda_1, \lambda_2, \lambda_3, \lambda_4)$ is again a smooth rational function satisfying $\mathcal{G}_2(\lambda_1, \lambda_1, \lambda_3, \lambda_4) = 0$. 

By analogous reasoning, as $\lambda_2 \to \lambda_1$, this term converges to a uniformly bounded limit as well:
\[
    \lim_{\lambda_2 \to \lambda_1} \frac{\mathcal{G}_2(\lambda_1, \lambda_2, \lambda_3, \lambda_4)}{\lambda_2 - \lambda_1} = \lim_{\lambda_2 \to \lambda_1} \frac{\mathcal{G}_2(\lambda_1, \lambda_2, \lambda_3, \lambda_4) - \mathcal{G}_2(\lambda_1, \lambda_1, \lambda_3, \lambda_4)}{\lambda_2 - \lambda_1} = \partial_{\lambda_2} \mathcal{G}_2(\lambda_1, \lambda_1, \lambda_3, \lambda_4).
\]
Consequently, the apparent singularity at $\lambda_2 = \lambda_1$ is also removable. There exists a $\delta_2 > 0$ such that this second good term is uniformly bounded on the set $\{0 < g < \delta_2\}$.

\vspace{2mm}

Since the constant scalar curvature $S$ ensures that all principal curvatures $\lambda_i$ are uniformly bounded on $M^4$, the good terms and all other non-singular terms remain uniformly bounded across the entire manifold.

\vspace{3mm}

\textbf{Step 2.2: (the bad terms).} The bad terms arising from the coefficients of $h_{441}^2$ and $h_{442}^2$ are given by
\[ - \tilde{m}_g \mathcal{W}_{12} \frac{\lambda_2}{\lambda_4} \frac{(\lambda_4 - \lambda_3)(\lambda_4 - \lambda_1)}{(\lambda_3 - \lambda_2)(\lambda_2 - \lambda_1)^2}\ \text{ and }\  - \tilde{m}_g \mathcal{W}_{12} \frac{\lambda_1}{\lambda_4} \frac{(\lambda_4 - \lambda_3)(\lambda_4 - \lambda_2)}{(\lambda_3 - \lambda_1)(\lambda_2 - \lambda_1)^2}.\]
However, since $\lambda_1\le\lambda_2<0<\lambda_3 < \lambda_4$, so we have
\begin{itemize}
    \item[(i)]\  $-\tilde{m}_g = -\dfrac{2(\lambda_4 - \lambda_3)}{\lambda_4} \left( \dfrac{\lambda_2(\lambda_4 - \lambda_1)}{\lambda_3 - \lambda_2} + \dfrac{\lambda_1(\lambda_4 - \lambda_2)}{\lambda_3 - \lambda_1} \right)>0$;\\[1mm]
    \item[(ii)]\   $\mathcal{W}_{12}  = \lambda_1^2 + \lambda_1 \lambda_2 + \lambda_2^2 - \frac{1}{2}S  =\frac{1}{2}(\lambda_1 + \lambda_2)^2 - \frac{1}{2}(\lambda_3^2 + \lambda_4 ^2)  = \lambda_3\lambda_4 >0$;\\[1mm]
    \item[(iii)]\  $\dfrac{\lambda_2}{\lambda_4} \dfrac{(\lambda_4 - \lambda_3)(\lambda_4 - \lambda_1)}{(\lambda_3 - \lambda_2)(\lambda_2 - \lambda_1)^2}<0,\ \  \dfrac{\lambda_1}{\lambda_4} \dfrac{(\lambda_4 - \lambda_3)(\lambda_4 - \lambda_1)}{(\lambda_3 - \lambda_2)(\lambda_2 - \lambda_1)^2}<0$
\end{itemize}
consequently, these bad terms are all strictly negative.

\vspace{3mm}

\textbf{Step 3: (estimate of $dg\wedge \varphi$).} 
As established above, the principal curvatures $\lambda_i$ are uniformly bounded, $h_{44k}^2 \ge 0$, and the bad terms are strictly negative. Consequently, there exist constants $\delta_g =\min \{\delta_1, \delta_2\} > 0$ and $C > 0$ (where $C$ depends only on $S$, $f_4$, and $\delta_g$) such that on the set $\{0 < g < \delta_g\}$, we have
\[dg \wedge \varphi \le C\left(\sum_{k=1}^4h_{44k}^2\right)\mathrm{vol}.\]

\vspace{3mm}

\textbf{Step 4: (similar argument for $df\wedge \varphi$).}  

Similarly, a local computation shows
\begin{align*}
    f_i & = 2(\lambda_4 - \lambda_3) (h_{44i} - h_{33i}) = 2(\lambda_4 - \lambda_3)\left[ h_{44i} + \left(\frac{\lambda_3}{\lambda_4}\right)\frac{(\lambda_4 - \lambda_1)(\lambda_4 - \lambda_2)}{(\lambda_3 - \lambda_1)(\lambda_3 - \lambda_2)}h_{44i}\right]\\[1.5mm]
    &:= (\lambda_4 - \lambda_3) \tilde{m}_f h_{44i},
\end{align*}
where $\tilde{m}_f > 0$ is a bounded factor. Denote $\mathcal{M}_f := (\lambda_4 - \lambda_3)^2\dfrac{\tilde{m}_f}{\lambda_4}$, then
\begin{align*}
 &df \wedge \varphi = \sum_{i=1}^4 f_i \omega_i \wedge \varphi \nonumber \\[2mm]
=&\  \mathcal{M}_f \left[ \frac{- \mathcal{W}_{12}\lambda_2(\lambda_4 - \lambda_1)}{(\lambda_3 - \lambda_2)(\lambda_2 - \lambda_1)^2} + \frac{\mathcal{W}_{13}\lambda_3(\lambda_4 - \lambda_2)(\lambda_4 - \lambda_1)}{(\lambda_3 - \lambda_2)(\lambda_3 - \lambda_1)^2 (\lambda_4 - \lambda_3)}
- \frac{\mathcal{W}_{14}\lambda_4}{(\lambda_4 - \lambda_1) (\lambda_4 - \lambda_3)} \right] h_{441}^2 \mathrm{vol} \\[2mm]
 & +  \mathcal{M}_f \left[ \frac{\mathcal{W}_{23}\lambda_3 (\lambda_4 - \lambda_2)(\lambda_4 - \lambda_1)}{(\lambda_3 - \lambda_1)(\lambda_3 - \lambda_2)^2(\lambda_4 - \lambda_3)}
- \frac{\mathcal{W}_{24}\lambda_4}{(\lambda_4 - \lambda_3)(\lambda_4 - \lambda_2)} -\frac{\mathcal{W}_{12}\lambda_1(\lambda_4 - \lambda_2)}{(\lambda_3 - \lambda_1)(\lambda_2 - \lambda_1)^2} \right] h_{442}^2 \mathrm{vol} \\[2mm]
 & +  \mathcal{M}_f  \left[ -  \frac{\mathcal{W}_{34}\lambda_4}{(\lambda_4 - \lambda_3)^2} + \frac{\mathcal{W}_{23}\lambda_2(\lambda_4 - \lambda_1)}{(\lambda_3 - \lambda_2)^2(\lambda_2 - \lambda_1)} - \frac{\mathcal{W}_{13}\lambda_1(\lambda_4 - \lambda_2)}{(\lambda_3 - \lambda_1)^2(\lambda_2 - \lambda_1)} \right] h_{443}^2 \mathrm{vol} \\[2mm]
 + &  \mathcal{M}_f \! \left[ \frac{- \mathcal{W}_{14}\lambda_1(\lambda_4\! -\! \lambda_2)}{(\lambda_3 \!- \!\lambda_1)(\lambda_2 \!-\! \lambda_1)(\lambda_4\! - \!\lambda_1)} \!+ \! \frac{\mathcal{W}_{24}\lambda_2(\lambda_4\! - \!\lambda_1)}{(\lambda_2\! -\! \lambda_1)(\lambda_3\! -\! \lambda_2)(\lambda_4\! - \!\lambda_2)} \!  - \! \frac{\mathcal{W}_{34}\lambda_3(\lambda_4 \! -\! \lambda_2)(\lambda_4 \! - \! \lambda_1)}{(\lambda_3 \! -\!  \lambda_1)(\lambda_3 \!-\! \lambda_2)(\lambda_4 \!- \!\lambda_3)^2} \!\right]\! h_{444}^2 \mathrm{vol}.
\end{align*}

\vspace{2mm}

Similarly to the argument of $dg \wedge \varphi$, the first good term arising from the coefficient of $h_{441}^2$ satisfies 
\begin{equation}\label{good3}
\lim_{\lambda_4 \to \lambda_3} \left( \frac{\mathcal{W}_{13}\lambda_3(\lambda_4 - \lambda_2)(\lambda_4 - \lambda_1)}{(\lambda_3 - \lambda_2)(\lambda_3 - \lambda_1)^2 (\lambda_4 - \lambda_3)}
- \frac{\mathcal{W}_{14}\lambda_4}{(\lambda_4 - \lambda_1) (\lambda_4 - \lambda_3)} \right) = \partial_{\lambda_4} \mathcal{F}_1(\lambda_1,\lambda_2,\lambda_3,\lambda_3), \end{equation}
and the second good term arising from the coefficient of $h_{442}^2$ satisfies
\begin{equation}\label{good4}\lim_{\lambda_4 \to \lambda_3} \left( \frac{\mathcal{W}_{23}\lambda_3 (\lambda_4 - \lambda_2)(\lambda_4 - \lambda_1)}{(\lambda_3 - \lambda_1)(\lambda_3 - \lambda_2)^2(\lambda_4 - \lambda_3)}
- \frac{\mathcal{W}_{24}\lambda_4}{(\lambda_4 - \lambda_3)(\lambda_4 - \lambda_2)} \right) =\partial_{\lambda_4}\mathcal{F}_2(\lambda_1, \lambda_2,\lambda_3,\lambda_3), \end{equation}
where $\mathcal{F}_1$ and $\mathcal{F}_2$ are smooth rational functions. Since each $\lambda_i$ is bounded on $M^4$, so the partial derivatives $\partial_4 \mathcal{F}_i$ are uniformly bounded.

Thus, these good terms, together with the remaining non-singular terms, are uniformly bounded.

\vspace{3mm}

On the other hand, the bad terms arising from the coefficients of $h_{443}^2$ and $h_{444}^2$ are given by
\[
 \frac{-\mathcal{W}_{34}\lambda_4}{(\lambda_4 - \lambda_3)^2}  \quad \text{and} \quad \frac{-\mathcal{W}_{34}\lambda_3(\lambda_4 - \lambda_2)(\lambda_4 - \lambda_1)}{(\lambda_3 - \lambda_1)(\lambda_3 - \lambda_2)(\lambda_4 - \lambda_3)^2}.
\]
Observe that
\[
\mathcal{W}_{34} = \lambda_3^2 + \lambda_3 \lambda_4 + \lambda_4^2 - \frac{1}{2}S = \frac{1}{2}(\lambda_3 + \lambda_4)^2 - \frac{1}{2}(\lambda_1^2 + \lambda_2 ^2) = \lambda_1\lambda_2 > 0,
\]
which implies that both of these bad terms are strictly negative.

\vspace{2mm}

As in the case of $dg\wedge \varphi$, we conclude that there exists a $\delta_3 > 0$ such that on the set $\{0 < f < \delta_3\}$, 
\[df \wedge \varphi \le C \sum_{k=1}^4(\lambda_4-\lambda_3)^2 h_{44k}^2.\]

Finally, by choosing $\delta := \min\{\delta_g, \delta_3\}$, we obtain a uniform constant $\delta > 0$ such that the estimates hold for both $f$ and $g$.

\end{proof}

\vspace{3mm}

Our further argument requires investigating the minimum function $\min\{f,g\}$. However, this function fails to be differentiable along the set $\{f-g = 0\}$. To resolve this non-smoothness, we employ the following smoothing lemma due to Guan \cite{G}:
\begin{lemma}[Guan \cite{G}]\label{lemma:Guan}
    For any $\eps > 0$, there exists an even function $h  \in C^{\infty}(\mathbb{R})$ such that
    \begin{align*}
    \mathrm{(i)}&\ h(t) \ge |t| \text{ for all } t\in \mathbb{R},\  h(t)=|t| \text{ for all }|t|\ge \eps;\\[1mm]
    \mathrm{(ii)} &\  |h'(t)|\le 1 \text{ and }h''(t) \ge 0  \text{ for all } t\in\mathbb{R}  \text{ and } h'(t)\ge 0 \text{ for all } t\ge 0.
\end{align*}
\end{lemma}

Because $M^4$ is compact and $f+g$ is strictly positive everywhere (since $f$ and $g$ cannot vanish simultaneously), there exists a sufficiently small constant $\epsilon_0 > 0$ such that $f+g > 2\epsilon_0$ across $M^4$.

\vspace{1mm}

Typically, we select the smoothing function $h$ with respect to $\eps=\epsilon_0$ and introduce the smoothed minimum function
\[u := \frac{f+g}{2} - \frac{h(|f-g|)}{2}.\]
Observe that $f$ and $g$ cannot be identically equal on $M^4$. If $f \equiv g$ everywhere, we would strictly have $f = g > 0$, which implies that $M^4$ possesses exactly four distinct principal curvatures everywhere, a contradiction to our assumption. 

Thus, the coincidence set $ E := \{f = g\}$ is a closed proper subset of $M^4$.

\vspace{2mm}

Setting $E_{\epsilon_0} := \{ |f - g| \le \epsilon_0\} \supset E$, we have the following properties of $u$:

\begin{itemize}
    \item[(i)] $u = \dfrac{f+g}{2}-\dfrac{|f-g|}{2} = \min\{f,g\}\ge 0$ on $M\setminus E_{\epsilon_0}$. Thus, $u = 0$ iff. $\min\{f,g\} = 0$. 
    \vspace{2mm}
    \item[(ii)] $u\ge \epsilon_0 - \dfrac{\epsilon_0}{2} \ge \dfrac{\epsilon_0}{2}$  on $E_{\epsilon_0}$,
\end{itemize} 

\vspace{1mm}

This construction securely avoids the singularity that would otherwise occur on the coincidence set $E$.

\vspace{5mm}

We are now in a position to prove Lemma \ref{3roots:S>12}.

\begin{proof}[Proof of Lemma \ref{3roots:S>12}]
Suppose for the sake of contradiction that $S > 12$. 

Let
\[ V = P_\epsilon \sqcup U_\epsilon \sqcup Q_\epsilon := \{0 < g < \epsilon\} \sqcup \{u > \epsilon\} \sqcup \{0 < f < \epsilon\}. \]

For $0 < \epsilon < \min\{\delta, \epsilon_0/2\}$ chosen sufficiently small, we introduce a smooth cut-off function $\eta_\epsilon : \mathbb{R} \to [0,1]$ such that
\begin{itemize}
    \item[(a)] $0 \le \eta_\epsilon \le 1$;

\vspace{1.5mm}
    
    \item[(b)] $\eta_\epsilon(t) = 0$ for $t \le \frac{\epsilon}{3}$;

\vspace{1.5mm}

    \item[(c)] $\eta_\epsilon(t) = 1$ for $t \ge \epsilon$;
    
\vspace{1.5mm}
    
    \item[(d)] $0 \le \eta_\epsilon'(t) \le O \left( \frac{1}{\epsilon} \right)$ for $t \in \left(\frac{\epsilon}{3}, \epsilon\right)$, and $\eta_\epsilon' \equiv 0$ everywhere else.
\end{itemize}

\vspace{2mm}

By Stokes' Theorem and the identity $d\big((\eta_\epsilon \circ u) \varphi\big) = (\eta_\epsilon \circ u)\, d\varphi + (\eta'_\epsilon \circ u)\, du \wedge \varphi$, we have
\[ \int_V (\eta_\epsilon \circ u)\, d\varphi + \int_V (\eta'_\epsilon \circ u)\, du \wedge \varphi = 0. \]
Therefore, 
\begin{align*}
0 &\ge \int_V (\eta_\epsilon \circ u) \left( \sum_{k=1}^4 \mathcal{L}_k h_{44k}^2 - \mathcal{R} \right) \mathrm{vol} = - \int_V (\eta'_\epsilon \circ u)\, du \wedge \varphi \\[2mm]
& = - \int_{P_\epsilon} (\eta'_\epsilon \circ g)\, dg \wedge \varphi - \int_{Q_\epsilon} (\eta'_\epsilon \circ f)\, df \wedge \varphi \\[2mm]
& \gtrsim -\frac{C}{\epsilon}\int_{P_\epsilon} \sum_{k=1}^4 h_{44k}^2\, \mathrm{vol} - \frac{C}{\epsilon} \int_{Q_\epsilon} \sum_{k=1}^4 (\lambda_4 - \lambda_3)^2 h_{44k}^2 \mathrm{vol}. 
\end{align*}

Since $S(S-4) = \sum h_{ijk}^2$ is constant, each $h_{ijk}$ is uniformly bounded on $M^4$. Hence, 
\begin{align*}
        |h_{11k}| &= \left|- \left( \frac{\lambda_1}{\lambda_4} \right) \frac{(\lambda_3 - \lambda_4)(\lambda_2 - \lambda_4)}{(\lambda_3 - \lambda_1)(\lambda_2 - \lambda_1)} h_{44k}\right| \quad \text{is bounded on } P_\epsilon.
\end{align*}
This implies that $|h_{44k}| \lesssim \sqrt{\epsilon}$ on $P_\epsilon$. Note also that $f=(\lambda_4 - \lambda_3)^2<\epsilon$ on $Q_\epsilon$. Consequently, 
\begin{equation}\label{5:calR}
    0 \ge \int_V(\eta_\epsilon \circ u) \left( \sum_{k=1}^4 \mathcal{L}_k h_{44k}^2 - \mathcal{R} \right) \mathrm{vol} \gtrsim -C\, \mathrm{vol}(P_\epsilon \sqcup Q_\epsilon).
\end{equation}

Notice that $\mathrm{vol}(P_\epsilon) \le \mathrm{vol}(M^4)$ and $\mathrm{vol}(Q_\epsilon) \le \mathrm{vol}(M^4)$, both of which are finite. Since $\epsilon > 0$ can be chosen arbitrarily small, taking the limit as $\epsilon \to 0^+$ yields
\[ \lim_{\epsilon\to 0^+} \mathrm{vol}(P_\epsilon) = \mathrm{vol}\left(\bigcap_{\epsilon>0}P_\epsilon\right) = \mathrm{vol}(\emptyset) = 0, \]
\[ \lim_{\epsilon\to 0^+} \mathrm{vol}(Q_\epsilon) = \mathrm{vol}\left(\bigcap_{\epsilon>0}Q_\epsilon\right) = \mathrm{vol}(\emptyset) = 0. \]
Thus, as $\epsilon \to 0^+$, the inequality \eqref{5:calR} becomes
\[ 0 = \int_{V} \left(\sum_{k=1}^{4} \mathcal{L}_k h_{44k}^2 - \mathcal{R}\right) \mathrm{vol}. \]
However, if $S > 12$, Subsection 4.2 implies that
\[ \mathcal{L}_k < 0 \quad \text{and} \quad \mathcal{R} > 0. \]
Since $V$ is a nonempty open set, $\mathrm{vol}(V) > 0$. Consequently, the equality above cannot hold, which yields the desired contradiction.
\end{proof}

\vspace{3mm}


Since we must have $S \le 12$, we will next show that $S = 12$, which implies it is impossible to have three distinct principal curvatures anywhere. Our strategy is to apply the exact same argument to $\Psi$. For the term $du \wedge \Psi$, we only need to change the weights from $\mathcal{W}_{ij}$ to $W_{ij}$.
\vspace{2mm}

Because the only difference lies in these weights, the good term arising from $h_{443}^2$ is now given by
\begin{align*}
    & \frac{{W}_{23}\lambda_2(\lambda_4 - \lambda_3)(\lambda_4 - \lambda_1)}{\lambda_4(\lambda_3 - \lambda_2)^2(\lambda_2 - \lambda_1)} - \frac{{W}_{13}\lambda_1(\lambda_4 - \lambda_3)(\lambda_4 - \lambda_2)}{\lambda_4(\lambda_3 - \lambda_1)^2(\lambda_2 - \lambda_1)}. 
\end{align*}
It follows directly from \eqref{good1} that this expression is uniformly bounded.

Similarly, the other good terms arising from $h_{444}^2$ correspond exactly to \eqref{good2}, simply replacing $\mathcal{W}_{24}$ and $\mathcal{W}_{14}$ with $W_{24}$ and $W_{14}$. As demonstrated previously, it is uniformly bounded as well.

\vspace{2.5mm}

The bad terms arising from $h_{441}^2$ and $h_{442}^2$ are now
\[ - \tilde{m}_g W_{12} \frac{\lambda_2}{\lambda_4} \frac{(\lambda_4 - \lambda_3)(\lambda_4 - \lambda_1)}{(\lambda_3 - \lambda_2)(\lambda_2 - \lambda_1)^2} \quad \text{and} \quad - \tilde{m}_g W_{12} \frac{\lambda_1}{\lambda_4} \frac{(\lambda_4 - \lambda_3)(\lambda_4 - \lambda_2)}{(\lambda_3 - \lambda_1)(\lambda_2 - \lambda_1)^2}.\]
Since 
\[W_{12} = \lambda_1^2 + \lambda_1 \lambda_2 + \lambda_2^2=\frac{1}{2}(\lambda_1 + \lambda_2)^2 + \frac{1}{2} \lambda_1^2 + \frac{1}{2}\lambda_2^2>0,\]
these bad terms remain strictly negative.
    
\vspace{2mm}

Similarly, \eqref{good3} and \eqref{good4} hold for $df\wedge \Psi$ as well. We simply replace the weights $\mathcal{W}_{ij}$ with $W_{ij}$. Furthermore, since $W_{34} = \frac{1}{2}(\lambda_3 + \lambda_4)^2 + \frac{1}{2} \lambda_3^2 + \frac{1}{2}\lambda_4^2>0$, the bad terms in $df \wedge \Psi$ remain strictly negative. Consequently, the results of Lemma \ref{lemma:dfdg} still hold. More precisely, for constants $\sigma> 0$ and $C>0$, the following inequalities remain valid:
 \begin{itemize}
     \item[(i)] on the set $\{ 0 < g < \sigma\}$,
$\displaystyle{dg \wedge \Psi \le C\left(\sum_{k=1}^4h_{44k}^2\right) \mathrm{vol}},$
\item[(ii)] on the set $\{0<f < \sigma\}$,
$\displaystyle{df \wedge \Psi\le C \left(\sum_{k=1}^4 (\lambda_4 - \lambda_3)^2 h_{44k}^2\right)\mathrm{vol}.}$
 \end{itemize}


Now, for any sufficiently small $0 < \epsilon < \min\{\sigma, \epsilon_0/2\}$, the key inequality \eqref{5:calR} becomes
\begin{align}\label{5:tilR}
    0 & \ge \int_V(\eta_\epsilon \circ u) \left( \sum_{k=1}^4 \tilde{L}_k h_{44k}^2 - \tilde{R} \right) \mathrm{vol} \nonumber\\[2mm]
     & =\int_V(\eta_\epsilon \circ u) \left( \sum_{k=1}^4 \tilde{L}_k h_{44k}^2 + \frac{1}{4}(S-12)(S-2) \right) \mathrm{vol} \nonumber\\[2mm]
     & \gtrsim -C\, \mathrm{vol}(P_\epsilon \sqcup Q_\epsilon).
\end{align}
As $\epsilon \to 0^+$, the same argument shows $\mathrm{vol}(P_\epsilon \sqcup Q_\epsilon) \to 0$. Consequently, this implies
\[\int_V \left( \sum_{k=1}^4 \tilde{L}_k h_{44k}^2 + \frac{1}{4}(S-12)(S-2) \right) \mathrm{vol} = 0.\]

\vspace{2mm}

By Lemma \ref{3roots:S>12}, we must have $12 \ge S > 4$. Furthermore, from Subsection 4.3, we know that $\tilde{L}_k < 0$. 

For \eqref{5:tilR} to hold as $\epsilon \to 0^+$, each $h_{44k}$ must vanish on $V$. Consequently, $h_{iik}\equiv 0$ on $V$ and $S\equiv 12$. 

\vspace{2mm}


Since the sets $P=\{g=0\}$ and $Q=\{f=0\}$ are closed, and $P\sqcup Q = V^c$, it follows that \[\partial V \subset P\sqcup Q.\]
By continuity, the principal curvatures $\lambda_1<\lambda_2<\lambda_3<\lambda_4$ are constant on the closure $\overline{V}$, which contains $\partial V$. This contradicts the assumption that there are exactly three distinct principal curvatures on $P \sqcup Q$.

\vspace{3mm}

\vspace{2mm}

We conclude this section with the following theorem:

\begin{theorem}\label{5:3buxing}
    Let $M^4\hookrightarrow \mathbb{S}^5(1)$ be a closed minimal hypersurface with constant scalar curvature and constant Gauss-Kronecker curvature. Then $M^4$ cannot have exactly three distinct principal curvatures at any point.
\end{theorem}

\vspace{1cm}

\addtocontents{toc}{\protect\setcounter{tocdepth}{2}}
\section{Proof of The Main Theorem}

\addtocontents{toc}{\protect\setcounter{tocdepth}{-1}} 
\subsection*{Proof of Theorem 1.11}
We will prove the main theorem by considering the following cases.

\subsection*{Case I} 
If the constant Gauss-Kronecker curvature is zero, then by Theorem \ref{Cui} (Cui \cite{Cui}), $M^4$ is totally geodesic. Therefore, $M^4$ is an equatorial $4$-sphere and $S=0$.

\subsection*{Case II} 
If the constant Gauss-Kronecker curvature is non-zero, we consider the following subcases. 

\subsubsection*{Case II.1} 
If there exists a point with exactly two distinct principal curvatures, then by Theorem \ref{theorem3}, $M^4$ is a Clifford torus $\mathbb{S}^2\left(\frac{\sqrt{2}}{2}\right)\times \mathbb{S}^2\left(\frac{\sqrt{2}}{2}\right)$ or $\mathbb{S}^1\left(\frac{1}{2}\right)\times \mathbb{S}^3\left(\frac{\sqrt{3}}{2}\right)$, and $S=4$.

\subsubsection*{Case II.2} 
If $M^4$ has exactly four distinct principal curvatures everywhere, then $M^4$ is isoparametric and $S=12$ by Theorem \ref{th4}. That is, $M^4$ is Cartan's minimal hypersurface.

\subsubsection*{Case II.3} 
If there exists a point with exactly three distinct principal curvatures, this contradicts Theorem \ref{5:3buxing} from Section 4.

\vspace{3mm}

The proof is complete.

\addtocontents{toc}{\protect\setcounter{tocdepth}{2}}

\newpage
\appendix

\section{The negativity of homogeneous polynomials $\mathcal{P}_k$ in \eqref{calLk}} 
Here, we provide the explicit expressions for $\mathcal{P}_k$ for $k=2,3,4$.

First,
\begin{align}\label{calP2}
    \mathcal{P}_2 = & \lambda_1^5+\lambda_2^5+5 \lambda_2^4 \lambda_3+9 \lambda_2^3 \lambda_3^2+7 \lambda_2^2 \lambda_3^3+3 \lambda_2 \lambda_3^4+\lambda_3^5+3 \lambda_1^4 (\lambda_2+\lambda_3) \nonumber\\[1.5mm]
    & +\lambda_1^3 (7 \lambda_2^2+14 \lambda_2 \lambda_3+5 \lambda_3^2) \nonumber +\lambda_1^2 (9 \lambda_2^3+27 \lambda_2^2 \lambda_3+23 \lambda_2 \lambda_3^2+5 \lambda_3^3) \nonumber\\[1.5mm]
    &+\lambda_1 (5 \lambda_2^4+20 \lambda_2^3 \lambda_3+27 \lambda_2^2 \lambda_3^2+14\lambda_2 \lambda_3^3+3 \lambda_3^4),
\end{align}

\vspace{3mm}

Similarly, calculate $\mathcal{P}_3$, we have 
\begin{align}\label{calP3}
    \mathcal{P}_3 =& \lambda_1^5 + \lambda_2^5 + 3 \lambda_2^4 \lambda_3 + 7 \lambda_2^3 \lambda_3^2 + 9 \lambda_2^2 \lambda_3^3 + 5 \lambda_2 \lambda_3^4 + \lambda_3^5 + 3 \lambda_1^4 (\lambda_2+\lambda_3) \nonumber\\[1.5mm]
    & + \lambda_1^3 (5 \lambda_2^2+14 \lambda_2 \lambda_3+7 \lambda_3^2) + \lambda_1^2 (5 \lambda_2^3+23 \lambda_2^2 \lambda_3+27 \lambda_2 \lambda_3^2+9 \lambda_3^3) \nonumber\\[1.5mm]
    &+ \lambda_1 (3 \lambda_2^4+14 \lambda_2^3 \lambda_3+27 \lambda_2^2 \lambda_3^2+20 \lambda_2 \lambda_3^3+5 \lambda_3^4),
 \end{align}
 
     \vspace{2mm}

Finally, for $\mathcal{P}_4$, recall that the factor $\alpha$ introduced in \eqref{solu} corresponds to exchanging indices $1$ and $4$. To be precise,
    \[h_{114}^2 =  \alpha {h_{444}^2} = \frac{\lambda_1^2(\lambda_3 - \lambda_4)^2(\lambda_2 - \lambda_4)^2}{\lambda_4^2(\lambda_3-\lambda_1)^2(\lambda_2-\lambda_1)^2}h_{444}^2,\] thus,
\begin{align}\label{calP4}
    \alpha \mathcal{P}_4 =& \lambda_1^5 + \lambda_2^5 + \lambda_2^4 \lambda_3 - \lambda_2^3 \lambda_3^2 - \lambda_2^2 \lambda_3^3 + \lambda_2 \lambda_3^4 + \lambda_3^5+ \lambda_1^4 (\lambda_2+\lambda_3) \nonumber \\[2mm]
    & - \lambda_1^3 (\lambda_2^2+\lambda_3^2) - \lambda_1^2 (\lambda_2^3+\lambda_2^2 \lambda_3+\lambda_2 \lambda_3^2+\lambda_3^3)+ \lambda_1 (\lambda_2^4-\lambda_2^2 \lambda_3^2+\lambda_3^4),
\end{align}
it shows that $\mathcal{P}_4 < 0$ as well.

\addtocontents{toc}{\protect\setcounter{tocdepth}{-1}} 
\subsection{The negativity on $U_1$}
Recall the change of variables introduced in \eqref{abcU1}. Since the substitution for $\mathcal{P}_1$ was already demonstrated in Section 4, we now apply these substitutions to $\mathcal{P}_2$, $\mathcal{P}_3$, and $\mathcal{P}_4$:
\begin{align*}
    \mathcal{P}_2 
    = & \ -26 b^5 - 61 b^4 c - 58 b^3 c^2 - 29 b^2 c^3 - 8 b c^4 - c^5 - 
 4 a^3 (3 b^2 + 3 b c + c^2) \nonumber \\[1.5mm] 
 & -2  a^2  (24 b^3 + 35 b^2 c + 19 b c^2 + 4 c^3) - 
 a (61 b^4 + 116 b^3 c + 87 b^2 c^2 + 32 b c^3 + 5 c^4); \nonumber\\[2.5mm]
    \mathcal{P}_3 =& -4 a^3 (b^2 + b c + c^2) - (2 b + c) (3 b^2 + 3 b c + c^2)^2 - 
 2 a^2 (10 b^3 + 15 b^2 c + 13 b c^2 + 4 c^3) \nonumber \\[1.5mm] 
 & - a (33 b^4 + 66 b^3 c + 61 b^2 c^2 + 28 b c^3 + 5 c^4); \nonumber\\[2.5mm]
    \alpha \mathcal{P}_4 =&  -2 b^5 - 5 b^4 c - 12 b^3 c^2 - 13 b^2 c^3 - 6 b c^4 - c^5 - 
     4 a^3 (b^2 + b c + c^2) \nonumber \\[1.5mm]
     &- 2 a^2 (6 b^3 + 9 b^2 c + 11 b c^2 + 4 c^3) - a (9 b^4 + 18 b^3 c + 31 b^2 c^2 + 22 b c^3 + 5 c^4), 
\end{align*}
Clearly, as with $\mathcal{P}_1$, all the coefficients are strictly negative; therefore, on $U_2$, we have
\[\mathcal{P}_k < 0, \quad k=1,2,3,4.\]

\subsection{The negativity on $U_2$}
Recall the change of variables introduced in \eqref{abcU2}. Substituting these into $\mathcal{P}_k$ for $k=1,2,3,4$, we obtain
\begin{align*}
    \mathcal{P}_1 =& -a^5 - 5 a^4 (b + c) - b^2 (b + c) (b + 2 c)^2 - 
 a^3 (9 b^2 + 18 b c + 8 c^2) \\[1.5mm]
     &- 
 a b (3 b^3 + 12 b^2 c + 14 b c^2 + 4 c^3) - 
 a^2 (7 b^3 + 21 b^2 c + 18 b c^2 + 4 c^3);\\[3mm]
    \mathcal{P}_2 =& -a^5 - b^2 (b + c) (3 b + 2 c)^2 - a^4 (7 b + 5 c) - 
 a^3 (21 b^2 + 24 b c + 8 c^2) \\[1.5mm]
     &- 
 a^2 (33 b^3 + 47 b^2 c + 22 b c^2 + 4 c^3) - 
 a b (27 b^3 + 48 b^2 c + 26 b c^2 + 4 c^3);\\[3mm]
    \mathcal{P}_3 =& -a^5 - a^4 (7 b + 5 c) - a^3 (23 b^2 + 28 b c + 8 c^2) - 
 a^2 (41 b^3 + 69 b^2 c + 34 b c^2 + 4 c^3) \\[1.5mm]
     &- 
 b^2 (13 b^3 + 37 b^2 c + 36 b c^2 + 12 c^3) - 
 a b (37 b^3 + 82 b^2 c + 58 b c^2 + 12 c^3);\\[3mm]
\alpha\mathcal{P}_4 =& -a^5 - a^4 (9 b + 5 c) - a^3 (31 b^2 + 34 b c + 8 c^2) - 
 a^2 (51 b^3 + 83 b^2 c + 38 b c^2 + 4 c^3) \\[1.5mm]
     & - 
 b^2 (13 b^3 + 37 b^2 c + 36 b c^2 + 12 c^3) - 
 a b (41 b^3 + 90 b^2 c + 62 b c^2 + 12 c^3),
\end{align*}
Similarly, all coefficients remain strictly negative. Hence, on $U_2$,
\[\mathcal{P}_k < 0, \quad k=1,2,3,4.\]
Consequently, $\mathcal{L}_k < 0$ for $k=1,2,3,4$ everywhere, since $\lambda_1 < \lambda_2 < 0 < \lambda_3 < \lambda_4$.

 \vspace{1.5mm}

\addtocontents{toc}{\protect\setcounter{tocdepth}{2}}

\section{The positivity of homogeneous polynomials $P_k$ in \eqref{tilLk} }
Here, we provide the explicit expressions for $P_k$ for $k=2,3,4$, as well as their corresponding forms after substituting the positive variables $a, b, c$ according to \eqref{abc}.
\begin{align}\label{tilL2}
P_2 =&\  \lambda_1^6 \big(\lambda_2 + 3 \lambda_3\big)^2 + 3 \lambda_1^5 \big(\lambda_2 + \lambda_3\big) \big(\lambda_2 + 3 \lambda_3\big)^2 + \lambda_2^2 \lambda_3^2 \big(\lambda_2 + \lambda_3\big)^2 \big(\lambda_2^2 + \lambda_2 \lambda_3 + \lambda_3^2\big) \nonumber\\[2mm]
&\quad + \lambda_1^4 \big(4 \lambda_2^4 + 21 \lambda_2^3 \lambda_3 + 66 \lambda_2^2 \lambda_3^2 + 93 \lambda_2 \lambda_3^3 + 36 \lambda_3^4\big) \nonumber\\[2mm]
&\quad + \lambda_1 \lambda_2 \lambda_3 \big(-2 \lambda_2^5 - 3 \lambda_2^4 \lambda_3 + 5 \lambda_2^3 \lambda_3^2 + 21 \lambda_2^2 \lambda_3^3 + 21 \lambda_2 \lambda_3^4 + 6 \lambda_3^5\big) \nonumber\\[2mm]
&\quad + \lambda_1^3 \big(3 \lambda_2^5 + 5 \lambda_2^4 \lambda_3 + 24 \lambda_2^3 \lambda_3^2 + 88 \lambda_2^2 \lambda_3^3 + 93 \lambda_2 \lambda_3^4 + 27 \lambda_3^5\big) \nonumber\\[2mm]
&\quad + \lambda_1^2 \big(\lambda_2^6 - 3 \lambda_2^5 \lambda_3 - 6 \lambda_2^4 \lambda_3^2 + 24 \lambda_2^3 \lambda_3^3 + 66 \lambda_2^2 \lambda_3^4 + 45 \lambda_2 \lambda_3^5 + 9 \lambda_3^6\big), 
\end{align}

\begin{align}\label{tilL3}
P_3 =&\  \lambda_1^6 \big(3 \lambda_2 + \lambda_3\big)^2 + 3 \lambda_1^5 \big(\lambda_2 + \lambda_3\big) \big(3 \lambda_2 + \lambda_3\big)^2 + \lambda_2^2 \lambda_3^2 \big(\lambda_2 + \lambda_3\big)^2 \big(\lambda_2^2 + \lambda_2 \lambda_3 + \lambda_3^2\big) \nonumber\\[1.5mm]
& + \lambda_1^4 \big(36 \lambda_2^4 + 93 \lambda_2^3 \lambda_3 + 66 \lambda_2^2 \lambda_3^2 + 21 \lambda_2 \lambda_3^3 + 4 \lambda_3^4\big) \nonumber\\[1.5mm]
 & + \lambda_1 \lambda_2 \lambda_3 \big(6 \lambda_2^5 + 21 \lambda_2^4 \lambda_3 + 21 \lambda_2^3 \lambda_3^2 + 5 \lambda_2^2 \lambda_3^3 - 3 \lambda_2 \lambda_3^4 - 2 \lambda_3^5\big) \nonumber\\[1.5mm]
& + \lambda_1^3 \big(27 \lambda_2^5 + 93 \lambda_2^4 \lambda_3 + 88 \lambda_2^3 \lambda_3^2 + 24 \lambda_2^2 \lambda_3^3 + 5 \lambda_2 \lambda_3^4 + 3 \lambda_3^5\big) \nonumber\\[1.5mm]
& + \lambda_1^2 \big(9 \lambda_2^6 + 45 \lambda_2^5 \lambda_3 + 66 \lambda_2^4 \lambda_3^2 + 24 \lambda_2^3 \lambda_3^3 - 6 \lambda_2^2 \lambda_3^4 - 3 \lambda_2 \lambda_3^5 + \lambda_3^6\big),
\end{align}

   Recall that for $k=4$, we introduced the exchange factor $\alpha$, thus
\begin{align}\label{tilL4}
\alpha P_4 =& \lambda_1^6 \big(\lambda_2 - \lambda_3\big)^2 + 3 \lambda_1^5 \big(\lambda_2 - \lambda_3\big)^2 \big(\lambda_2 + \lambda_3\big) + \lambda_2^2 \lambda_3^2 \big(\lambda_2 + \lambda_3\big)^2 \big(\lambda_2^2 + \lambda_2 \lambda_3 + \lambda_3^2\big) \nonumber\\[1.5mm]
 &- \lambda_1 \lambda_2 \big(\lambda_2 - \lambda_3\big)^2 \lambda_3 \big(2 \lambda_2^3 + 7 \lambda_2^2 \lambda_3 + 7 \lambda_2 \lambda_3^2 + 2 \lambda_3^3\big) \nonumber\\[1.5mm]
& + \lambda_1^3 \big(\lambda_2 - \lambda_3\big)^2 \big(3 \lambda_2^3 + 11 \lambda_2^2 \lambda_3 + 11 \lambda_2 \lambda_3^2 + 3 \lambda_3^3\big) \nonumber\\[1.5mm]
& + \lambda_1^4 \big(4 \lambda_2^4 + 5 \lambda_2^3 \lambda_3 - 6 \lambda_2^2 \lambda_3^2 + 5 \lambda_2 \lambda_3^3 + 4 \lambda_3^4\big) \nonumber\\[1.5mm]
 &+ \lambda_1^2 \big(\lambda_2^6 - 3 \lambda_2^5 \lambda_3 - 6 \lambda_2^4 \lambda_3^2 - 8 \lambda_2^3 \lambda_3^3 - 6 \lambda_2^2 \lambda_3^4 - 3 \lambda_2 \lambda_3^5 + \lambda_3^6\big).
 \end{align}
 
\addtocontents{toc}{\protect\setcounter{tocdepth}{-1}} 
\subsection{The positivity on $V_1$}
Recall the change of variables introduced in \eqref{abc} on $V_1$. Since the substitution for $P_1$ was explicitly detailed in Section 4, we now apply these same substitutions to $P_k$ for $k=2,3,4$. This gives:
\begin{align*}
P_2 = &\ 29952 a^8 + 960 a^7 (91 b + 30 c) + 
 b^2 (2 b^2 + 3 b c + c^2)^2 (3 b^2 + 3 b c + c^2) \nonumber\\[1.5mm]
&\quad+ 
 16 a^6 (6967 b^2 + 4695 b c + 657 c^2) + 
 24 a^5 (3400 b^3 + 3565 b^2 c + 999 b c^2 + 72 c^3) \nonumber\\[1.5mm]
&\quad+ 
 8 a^3 b (1410 b^4 + 2777 b^3 c + 1836 b^2 c^2 + 442 b c^3 + 
    21 c^4) \nonumber \\[1.5mm]
&\quad+ 
 4 a^4 (9415 b^4 + 13872 b^3 c + 6195 b^2 c^2 + 834 b c^3 + 27 c^4) \nonumber\\[1.5mm]
&\quad+ 
 2 a^2 b^2 (1076 b^4 + 2748 b^3 c + 2571 b^2 c^2 + 1050 b c^3 + 
    159 c^4) \nonumber\\[1.5mm]
&\quad+ 
 8 a b^2 (30 b^5 + 97 b^4 c + 123 b^3 c^2 + 77 b^2 c^3 + 24 b c^4 + 
    3 c^5),\end{align*}

\begin{align*}
 P_3 = &\ 248832 a^8 + 383616 a^7 (2 b + c) + 
 9 b^2 (2 b^2 + 3 b c + c^2)^2 (3 b^2 + 3 b c + c^2) \nonumber\\[1.5mm]
& + 
 80 a^6 (12879 b^2 + 12879 b c + 3041 c^2) + 
 24 a^5 (32802 b^3 + 49203 b^2 c + 23201 b c^2 + 3400 c^3) \nonumber\\[1.5mm]
& + 
 4 a^4 (93603 b^4 + 187206 b^3 c + 132261 b^2 c^2 + 38658 b c^3 + 
    3835 c^4) \nonumber\\[1.5mm]
& + 
 24 a^3 (4730 b^5 + 11825 b^4 c + 11130 b^3 c^2 + 4870 b^2 c^3 + 
    963 b c^4 + 64 c^5) \nonumber\\[1.5mm]
& + 
 24 a b (96 b^6 + 336 b^5 c + 474 b^4 c^2 + 345 b^3 c^3 + 
    136 b^2 c^4 + 27 b c^5 + 2 c^6) \nonumber\\[1.5mm]
& + 
 2 a^2 (10716 b^6 + 32148 b^5 c + 37803 b^4 c^2 + 22026 b^3 c^3 + 
    6519 b^2 c^4 + 864 b c^5 + 32 c^6),\end{align*}
  
 \begin{align*}
\alpha P_4 = &\ 62208 a^8 + 88128 a^7 (2 b + c) + 
 b^2 (2 b^2 + 3 b c + c^2)^2 (3 b^2 + 3 b c + c^2)  \nonumber\\[1mm]
& + 
 16 a^6 (13527 b^2 + 13527 b c + 3433 c^2) + 
 120 a^5 (1254 b^3 + 1881 b^2 c + 943 b c^2 + 158 c^3)  \nonumber\\[1.5mm]
& + 
 4 a^4 (16203 b^4 + 32406 b^3 c + 24063 b^2 c^2 + 7860 b c^3 + 
    931 c^4)  \nonumber\\[1.5mm]
&  + 
 8 a^3 (2214 b^5 + 5535 b^4 c + 5410 b^3 c^2 + 2580 b^2 c^3 + 
    587 b c^4 + 48 c^5)  \nonumber\\[1.5mm]
&  + 
 8 a b (36 b^6 + 126 b^5 c + 180 b^4 c^2 + 135 b^3 c^3 + 56 b^2 c^4 + 
    12 b c^5 + c^6)  \nonumber\\[1.5mm]
& + 
 2 a^2 (1500 b^6 + 4500 b^5 c + 5427 b^4 c^2 + 3354 b^3 c^3 + 
    1095 b^2 c^4 + 168 b c^5 + 8 c^6),
\end{align*}
it shows that in $V_1$, $P_k>0 \text{ for } k=1,2,3,4.$

\subsection{The positivity on  $V_2$}
On $V_2$, applying the change of variables defined in \eqref{abcV2}, we obtain
\begin{align*}
    \frac{P_1}{4} =  & 7488 a^8 + 48 a^7 (338 b + 305 c) + 
   b^2 c^2 (2 b^2 + 3 b c + c^2)^2 + 
   4 a^6 (3844 b^2 + 7142 b c + 2929 c^2)\\[1.5mm] 
   & + 
   12 a^5 (640 b^3 + 1934 b^2 c + 1789 b c^2 + 381 c^3) \\[1.5mm]
   & + 
   a^4 (2080 b^4 + 9616 b^3 c + 15876 b^2 c^2 + 7660 b c^3 + 
      931 c^4) \\[1mm]
      & + 
   4 a b c (2 b^5 + 27 b^4 c + 55 b^3 c^2 + 41 b^2 c^3 + 12 b c^4 + 
      c^5) \\[1.5mm]
      & + 8 a^3 (36 b^5 + 260 b^4 c + 737 b^3 c^2 + 619 b^2 c^3 + 
      172 b c^4 + 12 c^5) \\[1.5mm]
      & + 
   4 a^2 (4 b^6 + 54 b^5 c + 285 b^4 c^2 + 380 b^3 c^3 + 
      183 b^2 c^4 + 30 b c^5 + c^6);
\end{align*}

\begin{align*}
  \frac{P_2}{4} & = 7488 a^8 + 96 a^7 (169 b + 75 c) + b^2 c^2 (2 b^2 + 3 b c + c^2)^2 + 
 4 a^6 (3844 b^2 + 3210 b c + 657 c^2) \\[1.5mm]
      & + 
 12 a^5 (640 b^3 + 950 b^2 c + 315 b c^2 + 36 c^3) + 
 8 a^3 b (36 b^4 + 170 b^3 c + 303 b^2 c^2 + 148 b c^3 + 3 c^4)\\[1.5mm]
      &  + 
 4 a b^2 c (2 b^4 + 21 b^3 c + 40 b^2 c^2 + 27 b c^3 + 6 c^4) + 
 a^4 (2080 b^4 + 5424 b^3 c + 4260 b^2 c^2 + 492 b c^3 + 27 c^4) \\[1.5mm]
      & + 
 4 a^2 b^2 (4 b^4 + 42 b^3 c + 165 b^2 c^2 + 180 b c^3 + 57 c^4)  ;
\end{align*}

\begin{align*}
    \frac{P_3}{4} &= 62208 a^8 + 2592 a^7 (44 b + 37 c) + 
 b^2 c^2 (2 b^2 + 3 b c + c^2)^2 + 
 4 a^6 (21708 b^2 + 39042 b c + 15205 c^2) \\[1.5mm]
      &+ 
 12 a^5 (2928 b^3 + 8658 b^2 c + 7209 b c^2 + 1700 c^3) \\[1.5mm]
      &+ 
 a^4 (7968 b^4 + 35952 b^3 c + 49284 b^2 c^2 + 24684 b c^3 + 
    3835 c^4) \\[1mm]
      &+ 
 4 a b c (6 b^5 + 39 b^4 c + 72 b^3 c^2 + 55 b^2 c^3 + 18 b c^4 + 
    2 c^5)  \\[1.5mm]
      &
  + 8 a^3 (120 b^5 + 846 b^4 c + 1777 b^3 c^2 + 1452 b^2 c^3 + 
    473 b c^4 + 48 c^5) \\[1.5mm]
      &+ 
 4 a^2 (12 b^6 + 162 b^5 c + 537 b^4 c^2 + 660 b^3 c^3 + 
    345 b^2 c^4 + 72 b c^5 + 4 c^6);
\end{align*}

\begin{align*}
    \frac{\alpha P_4}{4} & = 15552 a^8 + 1296 a^7 (28 b + 17 c) + 
 b^2 c^2 (2 b^2 + 3 b c + c^2)^2 + 
 4 a^6 (8748 b^2 + 11502 b c + 3433 c^2)  \\[1.5mm]
      &+ 
 12 a^5 (1488 b^3 + 3258 b^2 c + 2151 b c^2 + 395 c^3)  \\[1.5mm]
      &+ 
 a^4 (5088 b^4 + 17232 b^3 c + 19332 b^2 c^2 + 7980 b c^3 + 
    931 c^4)  \\[1.5mm]
      &+ 
 4 a b c (6 b^5 + 33 b^4 c + 57 b^3 c^2 + 41 b^2 c^3 + 12 b c^4 + 
    c^5)  \\[1.5mm]
      &+ 8 a^3 (96 b^5 + 516 b^4 c + 911 b^3 c^2 + 645 b^2 c^3 + 
    172 b c^4 + 12 c^5)  \\[1.5mm]
      &+ 
 4 a^2 (12 b^6 + 126 b^5 c + 357 b^4 c^2 + 396 b^3 c^3 + 
    183 b^2 c^4 + 30 b c^5 + c^6).
\end{align*}
This demonstrates that on $V_2$, 
\[P_k > 0, \quad k=1,2,3,4.\]

\subsection{The positivity on $V_3$}
On $V_3$, applying the change of variables defined in \eqref{abcV3}, we obtain
\begin{align*}
    P_1 & = 62208 a^8 + 88128 a^7 (2 b + c) + 
 b^2 (2 b^2 + 3 b c + c^2)^2 (3 b^2 + 3 b c + c^2) \\[1.5mm]
      &+ 
 16 a^6 (13527 b^2 + 13527 b c + 3433 c^2) + 
 120 a^5 (1254 b^3 + 1881 b^2 c + 943 b c^2 + 158 c^3) \\[1.5mm]
      &+ 
 4 a^4 (16203 b^4 + 32406 b^3 c + 24063 b^2 c^2 + 7860 b c^3 + 
    931 c^4) \\[1.5mm]
      &+ 
 8 a^3 (2214 b^5 + 5535 b^4 c + 5410 b^3 c^2 + 2580 b^2 c^3 + 
    587 b c^4 + 48 c^5) \\[1.5mm]
      &+ 
 8 a b (36 b^6 + 126 b^5 c + 180 b^4 c^2 + 135 b^3 c^3 + 56 b^2 c^4 + 
    12 b c^5 + c^6) \\[1.5mm]
      &+ 
 2 a^2 (1500 b^6 + 4500 b^5 c + 5427 b^4 c^2 + 3354 b^3 c^3 + 
    1095 b^2 c^4 + 168 b c^5 + 8 c^6);
\end{align*}

\begin{align*}
    P_2 &= 248832 a^8 + 383616 a^7 (2 b + c) + 
 9 b^2 (2 b^2 + 3 b c + c^2)^2 (3 b^2 + 3 b c + c^2) \\[1.5mm]
      &+ 
 80 a^6 (12879 b^2 + 12879 b c + 3041 c^2) + 
 24 a^5 (32802 b^3 + 49203 b^2 c + 23201 b c^2 + 3400 c^3) \\[1.5mm]
      &+ 
 4 a^4 (93603 b^4 + 187206 b^3 c + 132261 b^2 c^2 + 38658 b c^3 + 
    3835 c^4) \\[1.5mm]
      &+ 
 24 a^3 (4730 b^5 + 11825 b^4 c + 11130 b^3 c^2 + 4870 b^2 c^3 + 
    963 b c^4 + 64 c^5) \\[1.5mm]
      &+ 
 24 a b (96 b^6 + 336 b^5 c + 474 b^4 c^2 + 345 b^3 c^3 + 
    136 b^2 c^4 + 27 b c^5 + 2 c^6) \\[1.5mm]
      &+ 
 2 a^2 (10716 b^6 + 32148 b^5 c + 37803 b^4 c^2 + 22026 b^3 c^3 + 
    6519 b^2 c^4 + 864 b c^5 + 32 c^6);
\end{align*}

\begin{align*}
    P_3 &= 29952 a^8 + 960 a^7 (91 b + 30 c) + 
 b^2 (2 b^2 + 3 b c + c^2)^2 (3 b^2 + 3 b c + c^2) \\[1.5mm]
      &+ 
 16 a^6 (6967 b^2 + 4695 b c + 657 c^2) + 
 24 a^5 (3400 b^3 + 3565 b^2 c + 999 b c^2 + 72 c^3)\\[1.5mm]
      & + 
 8 a^3 b (1410 b^4 + 2777 b^3 c + 1836 b^2 c^2 + 442 b c^3 + 
    21 c^4) \\[1.5mm]
      &+ 
 4 a^4 (9415 b^4 + 13872 b^3 c + 6195 b^2 c^2 + 834 b c^3 + 27 c^4) \\[1.5mm]
      &+ 
 2 a^2 b^2 (1076 b^4 + 2748 b^3 c + 2571 b^2 c^2 + 1050 b c^3 + 
    159 c^4) \\[1.5mm]
      &+ 
 8 a b^2 (30 b^5 + 97 b^4 c + 123 b^3 c^2 + 77 b^2 c^3 + 24 b c^4 + 
    3 c^5) ;
\end{align*}

\begin{align*}
    \alpha P_4 & = 29952 a^8 + 960 a^7 (91 b + 61 c) + 
 b^2 (2 b^2 + 3 b c + c^2)^2 (3 b^2 + 3 b c + c^2) \\[1.5mm]
      &+ 
 16 a^6 (6967 b^2 + 9239 b c + 2929 c^2) + 
 24 a^5 (3400 b^3 + 6635 b^2 c + 4069 b c^2 + 762 c^3) \\[1.5mm]
      &+ 
 8 a b (b + c)^2 (30 b^4 + 53 b^3 c + 35 b^2 c^2 + 10 b c^3 + c^4) \\[1.5mm]
      &+ 
 2 a^2 (b + c)^2 (1076 b^4 + 1556 b^3 c + 783 b^2 c^2 + 152 b c^3 + 
    8 c^4) \\[1.5mm]
      &+ 
 4 a^4 (9415 b^4 + 23788 b^3 c + 21069 b^2 c^2 + 7600 b c^3 + 
    931 c^4) \\[1.5mm]
      &+ 
 8 a^3 (1410 b^5 + 4273 b^4 c + 4828 b^3 c^2 + 2504 b^2 c^3 + 
    587 b c^4 + 48 c^5).
\end{align*}
This shows that on $V_3$, $P_k>0\ \text{for }k=1,2,3,4.$

\subsection{The positivity on $V_4$}
On $V_4$, applying the change of variables defined in \eqref{abcV4}, we obtain

\begin{align*}
    \frac{P_1}{4} &= 15552 a^8 + 1296 a^7 (28 b + 17 c) + 
 b^2 c^2 (2 b^2 + 3 b c + c^2)^2 + 
 4 a^6 (8748 b^2 + 11502 b c + 3433 c^2) \\[1.5mm]
      &+ 
 12 a^5 (1488 b^3 + 3258 b^2 c + 2151 b c^2 + 395 c^3) \\[1.5mm]
      &+ 
 a^4 (5088 b^4 + 17232 b^3 c + 19332 b^2 c^2 + 7980 b c^3 + 
    931 c^4) \\[1.5mm]
      &+ 
 4 a b c (6 b^5 + 33 b^4 c + 57 b^3 c^2 + 41 b^2 c^3 + 12 b c^4 + 
    c^5) \\[1.5mm]
      &+ 8 a^3 (96 b^5 + 516 b^4 c + 911 b^3 c^2 + 645 b^2 c^3 + 
    172 b c^4 + 12 c^5) \\[1.5mm]
      &+ 
 4 a^2 (12 b^6 + 126 b^5 c + 357 b^4 c^2 + 396 b^3 c^3 + 
    183 b^2 c^4 + 30 b c^5 + c^6);
\end{align*}

\begin{align*}
    \frac{P_2}{4} &= 62208 a^8 + 2592 a^7 (44 b + 37 c) + 
 b^2 c^2 (2 b^2 + 3 b c + c^2)^2 + 
 4 a^6 (21708 b^2 + 39042 b c + 15205 c^2) \\[1.5mm]
      &+ 
 12 a^5 (2928 b^3 + 8658 b^2 c + 7209 b c^2 + 1700 c^3) \\[1.5mm]
      &+ 
 a^4 (7968 b^4 + 35952 b^3 c + 49284 b^2 c^2 + 24684 b c^3 + 
    3835 c^4) \\[1.5mm]
      &+ 
 4 a b c (6 b^5 + 39 b^4 c + 72 b^3 c^2 + 55 b^2 c^3 + 18 b c^4 + 
    2 c^5) \\[1.5mm]
      &+ 
 8 a^3 (120 b^5 + 846 b^4 c + 1777 b^3 c^2 + 1452 b^2 c^3 + 
    473 b c^4 + 48 c^5) \\[1.5mm]
      &+ 
 4 a^2 (12 b^6 + 162 b^5 c + 537 b^4 c^2 + 660 b^3 c^3 + 
    345 b^2 c^4 + 72 b c^5 + 4 c^6);
\end{align*}
\begin{align*}
    \frac{P_3}{4} & = 7488 a^8 + 96 a^7 (169 b + 75 c) + b^2 c^2 (2 b^2 + 3 b c + c^2)^2 + 
 4 a^6 (3844 b^2 + 3210 b c + 657 c^2) \\[1.5mm]
      &+ 
 12 a^5 (640 b^3 + 950 b^2 c + 315 b c^2 + 36 c^3) + 
 8 a^3 b (36 b^4 + 170 b^3 c + 303 b^2 c^2 + 148 b c^3 + 3 c^4) \\[1.5mm]
      &+ 
 4 a b^2 c (2 b^4 + 21 b^3 c + 40 b^2 c^2 + 27 b c^3 + 6 c^4) + 
 a^4 (2080 b^4 + 5424 b^3 c + 4260 b^2 c^2 + 492 b c^3 + 27 c^4) \\[1.5mm]
      &+ 
 4 a^2 b^2 (4 b^4 + 42 b^3 c + 165 b^2 c^2 + 180 b c^3 + 57 c^4);
\end{align*}

\begin{align*}
    \frac{\alpha P_4}{4} &= 7488 a^8 + 48 a^7 (338 b + 305 c) + b^2 c^2 (2 b^2 + 3 b c + c^2)^2 + 
 4 a^6 (3844 b^2 + 7142 b c + 2929 c^2) \\[1.5mm]
      &+ 
 12 a^5 (640 b^3 + 1934 b^2 c + 1789 b c^2 + 381 c^3) \\[1.5mm]
      &+ 
 a^4 (2080 b^4 + 9616 b^3 c + 15876 b^2 c^2 + 7660 b c^3 + 931 c^4) \\[1.5mm]
      &+ 
 4 a b c (2 b^5 + 27 b^4 c + 55 b^3 c^2 + 41 b^2 c^3 + 12 b c^4 + 
    c^5) \\[1.5mm]
      &+ 8 a^3 (36 b^5 + 260 b^4 c + 737 b^3 c^2 + 619 b^2 c^3 + 
    172 b c^4 + 12 c^5) \\[1.5mm]
      &+ 
 4 a^2 (4 b^6 + 54 b^5 c + 285 b^4 c^2 + 380 b^3 c^3 + 183 b^2 c^4 + 
    30 b c^5 + c^6).
\end{align*}
This demonstrates that on $V_4$, \[P_k > 0, \quad k=1,2,3,4.\]

\vspace{2mm}

Consequently, $\tilde{L}_k < 0$ for $k=1,2,3,4$, whenever $\lambda_1 < \lambda_2 < 0 < \lambda_3 < \lambda_4$. 

\vspace{3mm}

All of these explicit computations substantiate our analysis in Subsections 4.2 and 4.3.

\end{document}